\documentclass[10pt,a4paper,leqno]{amsart}

\usepackage{graphicx}
\usepackage{dsfont,epstopdf}
\usepackage{amsmath, amsthm, amssymb, amsfonts,mathrsfs,comment}
\usepackage{enumerate}
\usepackage{upgreek}
\usepackage{mathabx}
\usepackage[x11names,rgb,cmyk]{xcolor}
\usepackage{subfig}
\usepackage{tikz}
\usetikzlibrary{arrows,shapes,matrix,decorations.pathmorphing}

\usepackage{tabularx}
\usepackage{pgfplots}
\usetikzlibrary{arrows}
\usepackage{color}
\usepackage{tikz}
\usepackage{marginnote}

\usepackage[normalem]{ulem}
\theoremstyle{plain}
\newtheorem{thm}{Theorem}[section]
\newtheorem{Def}[thm]{Definition}
\newtheorem{lem}[thm]{Lemma}
\newtheorem{Prop}[thm]{Property}
\newtheorem{cor}[thm]{Corollary}

\newtheorem{rem}[thm]{Remark}

\newtheorem*{def*}{Definition}
\newtheorem*{thm*}{Theorem}

\renewcommand{\epsilon}{\varepsilon}
\renewcommand{\theta}{\vartheta}
\renewcommand{\phi}{\varphi}
\renewcommand{\pi}{\uppi}
\renewcommand{\delta}{\updelta}

\DeclareMathOperator{\fff}{F}

\DeclareMathOperator{\gagaga}{STFT}

\DeclareMathOperator{\ststst}{S}

\DeclareMathOperator{\supp}{supp}
\DeclareMathOperator{\Span}{span}

\newcommand{\pt}[1]{\left( #1 \right) }
\newcommand{\ptg}[1]{\left\{ #1 \right\} }
\newcommand{\ptq}[1]{\left[ #1 \right] }

\newcommand{\abs}[1]{\left\lvert #1\right\rvert}
\newcommand{\norm}[1]{\left\|  #1  \right\| }

\newcommand{\Lprod}[2]{ \langle #1, #2\rangle }

\newcommand{\Stu}[1]{\ststst #1  }

\def\Rr{\mathbb R}

\def\R{\mathbb R}
\def\Rn{{\mathbb R}^d}

\def\Z{{\mathbb Z}}
\def\N{{\mathbb N}}

\def\Ll{L^2 \pt{\Rr}}

\def\Lln{L^2 \pt{\Rn}}

\def\S{{\mathcal S}}
\def\F{\fff}

\def\dnbt{D_{p,\tau}}

\def\fdnbt{f_{p,\tau}}

\def\Lloi{L^2\pt{[0,1]}}

\def\pii{2\pi \mathrm{i}\,}
\def\ud{\mathrm{d}}
\def\Bpt{B^{\alpha}_{p, \tau}}
\def\Bpg{\varphi^{\alpha}_{p,k}}
\def\Bpp{\varphi_{\gamma}}
\def\Bapi{X_{p, \ell}}
\def\iap{i_{\alpha;p}}
\def\sap{s_{\alpha;p}}
\def\bap{\beta_{\alpha}\pt{p}}
\def\bpp{\beta\pt{p}}
\def\bppbar{\beta\pt{\bar{p}}}
\def\Iap{I_{\alpha;p}}
\def\Iapi{I^{\ell}_{p}}
\def\Iapibar{I^{\bar{\ell}}_{\bar{p}}}
\def\supp{\mathrm{supp\,}}
\newcommand{\ia}[1]{i_{\alpha;{#1}}}
\newcommand{\sa}[1]{s_{\alpha;{#1}}}
\newcommand{\ba}[1]{\beta_{\alpha}(#1)}
\newcommand{\Ia}[1]{I_{\alpha;{#1}}}
\def\Gapm{\Phi^\mu_{p;\alpha}}
\def\Gppm{\Phi^{\mu}_{p;\ell}}
\def\Gppmbar{\Phi^{\mu}_{\bar{p};\bar{\ell}}}
\def\PGapm{\Psi^\mu_{p;\alpha}}

\def\Bapkmn{\varphi^{\alpha;\mu, \nu}_{p,k}}
\def\Famnfg{\S^{\alpha;\mu,\nu}_{\varphi,\varphi}}
\def\PBapkmn{\psi^{\alpha;\mu, \nu}_{p,k}}
\def\PFamnfg{\S^{\alpha;\mu,\nu}_{\varphi,\psi}}
\def\vpt{\varphi^{\alpha}_{p,k}}
\newcommand\dist[2]{\mathrm{d}\pt{#1,#2}}

\numberwithin{equation}{section}
\begin{document}
\begin{abstract}
  We show the existence of a family of frames of $L^2(\R)$ which depend on
  a parameter $\alpha\in [0,1]$. If $\alpha=0$, we recover the usual Gabor frame,
  if $\alpha=1$ we obtain a frame system which is closely related to the so called 
  DOST basis, first introduced in \cite{ST07} and then analyzed in \cite{Battisti2015}.
  If $\alpha\in (0,1)$, the frame system is associated to a so called $\alpha$-partitioning
  of the frequency domain. 
  Restricting to the case $\alpha=1$, we provide a truly $n$-dimensional version of the DOST basis
  and an associated frame of $L^2(\Rr^d)$.
\end{abstract}
\vspace{1em}
%%%%%%%%%%%%%%%%%%%%%%%%%%%%%%%%%%%%%%%%%%%%%%%%%%%%%%%%%%%%%%%%%%%%%%%%%%%%%%%%%%%%%%%%%%%%%%%%%%%%%%%%%%%%%%%%%%%%%%%%%%%%
%%%%%%%%%%%%%%%%%%%%%%%%%%%%%%%%%%%%%%%%%%%%%%%%%%%%%%%%%%%%%%%%%%%%%%%%%%%%%%%%%%%%%%%%%%%%%%%%%%%%%%%%%%%%%%%%%%%%%%%%%%%%
%title

\title{Explicit construction of non-stationary frames for $L^2$}

\author{Ubertino Battisti, Michele Berra}
\address{Dipartimento di Scienze Matematiche,
Politecnico di Torino, corso Duca degli Abruzzi 24, 10129 Torino,
Italy}
\email{ubertino.battisti@polito.it}
\address{Universit\`a di Torino, Dipartimento di Matematica, via Carlo Alberto 10, 10123 Torino, Italy}
\email{michele.berra@unito.it}

\maketitle
%%%%%%%%%%%%%%%%%%%%%%%%%%%%%%%%%%%%%%%%%%%%%%%%%%%%%%%%%%
%%%%%%%%%%%%%%Introduzione%%%%%%%%%%%%%%%%
%%%%%%%%%%%%%%%%%%%%%%%%%%%%%%%%%%%%%%%%%%%%%%%%%%%%%%%%%%

\section{Introduction}
One of the most intriguing problems of modern Time-Frequency analysis is the construction of new efficient methods to 
represent signals, which can be one dimensional or, more often, multidimensional, such as digital images. 
The increasing amount of data and their complexity forces to develop optimized techniques that address the representation in a fast and efficient way. 

The starting point of this paper is the definition of the $\ststst$-transform, first introduced by R. G. Stockwell \emph{et al.}
in \cite{ST96} as
\begin{align}
  \label{eq:defStrans}
  (\ststst f)(b, \xi)= \pt{2\pi}^{-\frac{1}{2}}\int e^{ -\pii t \xi} |\xi| f(t) e^{- \frac{\xi^2(t-b)^2}{2}} dt.
\end{align}
 The expression of the $\ststst$-transform
\eqref{eq:defStrans} is formally very similar to the so called Short Time Fourier Transform or Gabor Transform 
($\gagaga$) with Gaussian window
\[
  \pt{\gagaga f}\pt{b, \xi}= \pt{2\pi}^{-\frac{1}{2}}\int e^{ -\pii t \xi} f(t) e^{- \frac{(t-b)^2}{2}} dt.
\]
The main novelty of the $\ststst$-transform  is the frequency depending window. The leading idea is  the heuristic fact that, in order to detect high frequencies, it is enough to consider a shorter time. Therefore, the width of the Gaussian is not fixed but depends on the frequency, shirking as far as the frequency increases. 
The $\ststst$-transform was introduced to improve the analysis of seismic imaging, and it is now considered an important
tool in geophysics, see \cite{Geo15}. In  \cite{MR2666947}, the connection between the phase of
$\ststst$-transform and the instantaneous frequency - useful in several applications - has been studied.
See  \cite{BIDA13,CH14,Drabycz:2009lq,GZBM04,BH13,LA14,GA13,ZU13} for some applications of the $\ststst$-transform to signal processing in general. 

From the mathematical point of view, M. W. Wong and  H. Zhu in \cite{WZ07} introduced a generalized version of the
$\ststst$-transform as follows
\[
  \pt{\Stu_{\phi}{f}}\pt{b,\xi} =\int_{\Rr}
e^{-\pii t\xi}\,f\pt{t}\abs{\xi} \overline{\phi\pt{\xi\pt{t-b}}}\, dt,\qquad b,\xi\in\Rr,
\]
where $\phi$ is a general window function in $L^2(\R)$. The $\ststst$-transform
has a strong similarity with the $\gagaga$, actually it is also possible to show a deep link with the wavelet transform, see  \cite{GLM06,VE08}.
In fact, the $\ststst$-transform can be seen as 
an hybrid between  the $\gagaga$ and the wavelet transform. Representation Theory provides a very deep connection among
$\ststst$-transform, $\gagaga$ and wavelet transform, as they all relate to the representation
of the so-called affine Weyl-Heisenberg group, studied in \cite{kalisa1993n}. 
This connection has been highlighted in the multi-dimensional case by L. Riba in \cite{RI14}, see also \cite{RW15}.  
The affine Weyl-Heisenberg group is also the key to represent the $\alpha$-modulation groups, \cite{DA08}.

Our analysis focuses on the DOST (Discrete Orthonormal Stockwell Transform), a discretization of the S-Transform, first introduced by R. G. Stockwell in \cite{ST07}.
In \cite{Battisti2015},
the DOST transform has been studied from a mathematical point of view. It is shown that the
DOST is essentially the decomposition of a periodic signal $f\in L^2([0,1])$
in a suitable orthonormal basis defined as
\begin{align}
\label{dostbasis}\bigcup_{p\in\Z} D_p & = \bigcup_{p\in\Z} \ptg{\dnbt}_{\tau = 0}^{2^{p-1}},
\end{align}
where
\begin{align*}
\dnbt (t) & =\frac{1}{\sqrt{2^{p-1}}}\sum_{\eta=2^{p-1}}^{2^{p}-1} e^{\pii \eta \pt{t-\frac{\tau}{2^{p-1}}}},\qquad t\in\Rr, p\geq1, \tau=0, \ldots, 2^{p-1}-1,    
\end{align*}
with the convention that $D_0(t) = 1$, see Section \ref{S:alfamod} for the precise description of the basis functions.
The DOST basis has a non stationary time-frequency localization, roughly speaking the time localization increases as
the frequency increases, while the frequency localization decreases as the frequency increases. Therefore,
the basis decomposition of a periodic signal is able to localize high frequencies, for example spikes.
The time-frequency localization properties of the DOST basis imply that the coefficients
\[
  \fdnbt = \pt{f, \dnbt}_{\Lloi}, 
\]
represent the time-frequency content of the signal $f$ in a certain time-frequency box, which is 
related to a dyadic decomposition in the frequency domain. 
Moreover, this decomposition can be seen as a sampling of a generalized $\ststst$-transform with a particular analyzing window
which is essentially a box car window in the frequency domain. 

The DOST transform gained interest in the applied world after the FFT-fast algorithm discovered by Y. Wang ang J. Orchard, see \cite{WAth, 
wang2009fast} for the original algorithm and \cite{Battisti2015} for a slightly different approach which shows that
the fast algorithm is essentially a clever application of Plancharel Theorem.

The orthogonality property is clearly very useful for several applications,
nevertheless it is well known that, in order to describe signals, a certain amount of redundancy can be very effective.
Therefore, it is a natural task to look for frames associated to the DOST basis. 

We follow the idea of the construction of Gabor frames:
\[
  \mathcal{G}(g, \alpha, \beta)=\ptg{ T_{\alpha k}M_{\beta n} g}= \ptg{e^{\pii \beta n( t-\alpha k)} g\pt{t- \alpha k}}_{(k,n)\;\in \Z^2}.
\]
If $g$ is a suitable window function, for example a Gaussian, and $\alpha \beta<1$ then $\mathcal{G}(g, \alpha, \beta)$
is a frame, see for example \cite{GR01}. It is possible to consider the Gabor frame as the frame associated to the standard
Fourier basis, which is formed by all modulations with integer frequency,  extended by periodicity and
then localized using the translation of the analyzing window function $g$. 
Inspired by this approach, we consider the system of functions
\begin{align}
  \label{eq:sist}
  \ptg{T_{k \frac{\nu}{\beta(p)}  } \pt{D_{p, 0}(t) g(t)}}_{(p,k)\; \in \Z^2}, \quad \beta(p)= 2^{|p|-1}\mbox{ if }p\neq0, \beta(0)=1, \nu>0.
\end{align}
For simplicity, here we have not introduced a frequency parameter. The idea of the system in \eqref{eq:sist}
is to consider the DOST basis $D_{p, 0}$ and then localize it with a window function $g$.
The main difference from the standard Gabor system is that the translation parameter $k \frac{\nu}{\beta(p)}$ is not uniform, but 
depends on the frequency parameter $p$. Therefore, \eqref{eq:sist} can be considered as 
a non stationary Gabor frame, using the terminology introduced in \cite{NSGF11}.
%The above construction could be repeated for other bases of $L^2([0,1])$, not only for the Fourier Basis or the DOST basis.
Inspired by the theory of   $\alpha$-modulation frames, see \cite{Feichtinger:2006rt}, \cite{FG85}, \cite{Fornasier2007157},               \cite{MR2714948}; in Section \ref{S:alfamod} we introduce a family of bases of $L^2([0,1])$
depending on a parameter $\alpha \in [0,1]$. If $\alpha=0$ we recover the standard Fourier basis, if $\alpha=1$ the DOST basis; when $\alpha \in (0,1)$ we show that the basis is associated to a suitable $\alpha$-partitioning
of the frequency domain in the sense of \cite{Fornasier2007157}.

In Section \ref{sec:dim1}, we prove the main result, see Theorem \ref{T:Frame_prop}; we show that for each $\alpha \in [0,1]$ the localization procedure explained above produces 
frames of $L^2(\R)$, provided the time and frequency parameter are small enough.
The main tool is a non stationary version of the Walnut representation, see Subsection \ref{subsec:Walnut}.

In Section \ref{sec:dimd}, we analyze the higher dimensional case. 
Restricting to $\alpha=1$, we consider a multidimensional partitioning 
of the frequency domain and the associated frames. This construction is different to the usual extension
of the DOST to higher dimensions. In \cite{WAth,wang2009fast}, the DOST applied to two dimensional signals (digital images)
was essentially the one dimensional DOST applied in the vertical and then in the horizontal direction, therefore 
it was not a truly bi-dimensional version of the transform.

% 
% The idea of our paper is to generalize this construction from two point of view: first, we generalize the time-frequency tiling introducing an $\alpha$-covering.
% This particular partition of the phase-space has been introduced with the theory of $\alpha$-modulation spaces by Grobner \cite{MR2714948} and L. P \"aiv\"arinta, E. Somersalo \cite{MANA:MANA19881380111}.\\
% The second generalization is the construction of a frame for $\Ll$  using the DOST as starting point; we first address the $1$-dimensional case for $\alpha \in \ptq{0,1}$; for $\alpha=1$ we also build a frame structure in arbitrary dimensions.\medskip\\
% 

% 
% \bibliography{Bib_Stock}
% %\nocite{*}
% \bibliographystyle{siam} 
% 
% \end{document}
\section{Notations}
In the paper we use the following normalization of the Fourier Transform
\[
  \F(f)(\omega)= \hat{f}(\omega)=\int e^{-2 \pi i x \omega} f(x) dx.
\]
In the multidimensional case, we do not write explicitly the inner product in $\R^d$, 
we leave the same notation of the one dimensional case.

Set
\[
  M_{k} f(x)= e^{2 \pi i k x} f(x),\quad k, x \in \R^d.
\]
and
\[
  T_k f(x)= f\pt{x-k}, \quad  k,x \in \R^d.
\]
Let $f,g\in L^2(\R^d)$, we denote the $L^2$-scalar product as
\[
  \langle f, g \rangle= \int f \bar{g} dx.
\]
We denote $\S(\R^d)$ the Schwartz space. 
%%%%%%%%%%%%%%%%%%%%%%%%%%%%%%%%%%%%%%%%%%%%%%%%%%%%%%%%%%
%%%%%%%%%%%%%%PARTE SULLA BASE ORTONORMALE%%%%%%%%%%%%%%%%
%%%%%%%%%%%%%%%%%%%%%%%%%%%%%%%%%%%%%%%%%%%%%%%%%%%%%%%%%%
 \section{$\alpha$-bases of $\Lloi$}
 \label{S:alfamod}
 Let $\alpha \in [0,1]$, we define a partition of $\N$
 associated to the parameter $\alpha$. We use an iterative approach.
 Let $p$ be a non negative integer. 
 For each $\alpha$, we define
 \begin{equation}\label{eq:bap}
   \begin{array}{llll} 
    \ia 0=0,       &\sa 0=1, & \ba 0= 1,  & \\
    \iap=s_{\alpha;p-1}, & \sap= \iap+\lfloor \iap^\alpha\rfloor      & \bap=
    \lfloor\iap^\alpha\rfloor,
    &\quad p\geq 1,
   \end{array} 
 \end{equation}
 where $\lfloor x\rfloor$ is the integer part of $x$.
 Then, set
 \begin{align*} 
  \Iap=[\iap, \sap)\cap \N & \quad p\text{th interval of the partition of } \N,\\
  \bap=\abs{\Iap}= \abs{\sap-\iap}&\quad \text{width of the interval }I_p.
 \end{align*}

 We then set for all $p\in \N$
 \begin{align*}
   \Bpt(t)= \frac{1}{\sqrt{\bap}}\sum_{\eta=\iap}^{\sap-1} 
   e^{\pii \eta \pt{t- \frac{\tau}{\bap}}}, \quad \tau=0, \ldots, \bap-1.
 \end{align*}
 See Figure~\ref{F: Dost_elements} for a plot of real and imaginary part of such basis with different values of $\alpha$ and $p$.
 \begin{figure}
\begin{tabular}{ccc}
\subfloat[$\alpha=0,~p=1$]{\includegraphics[width=.27\linewidth]{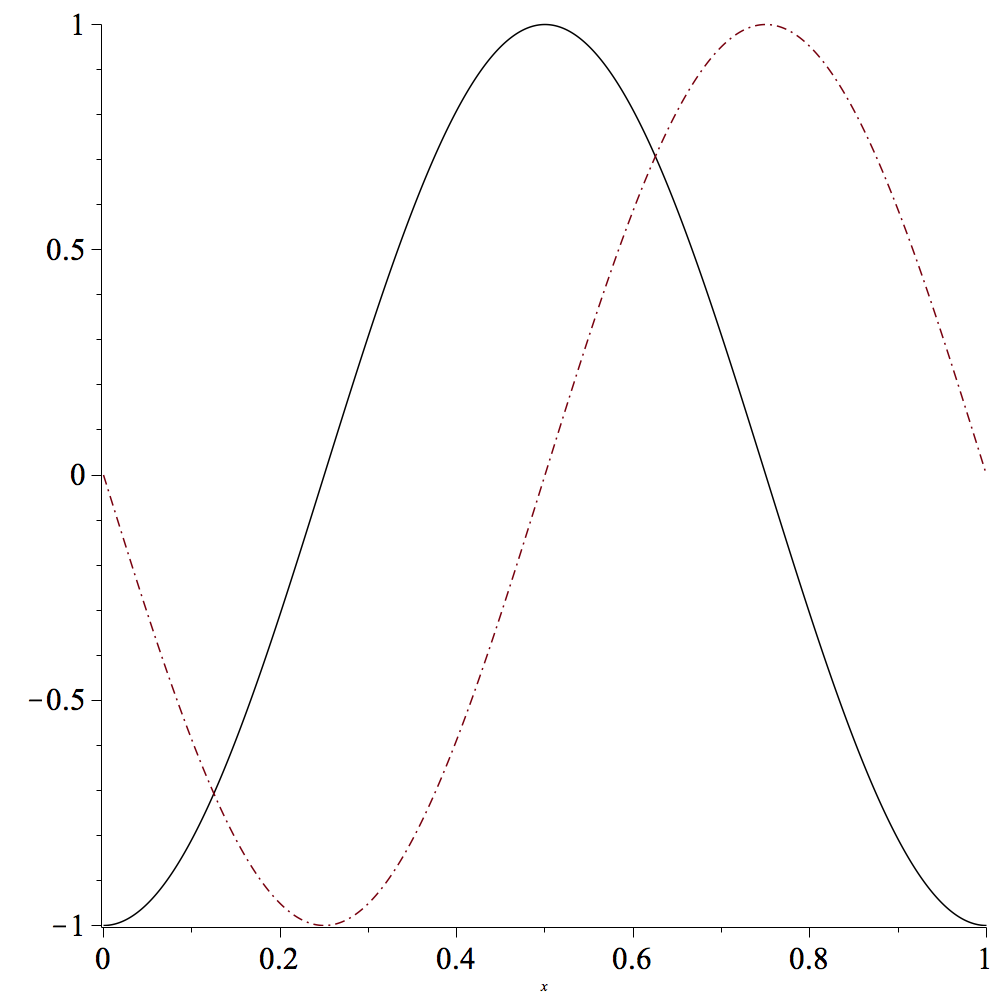}} &
\subfloat[$\alpha=0,~p=5$]{\includegraphics[width=.27\linewidth]{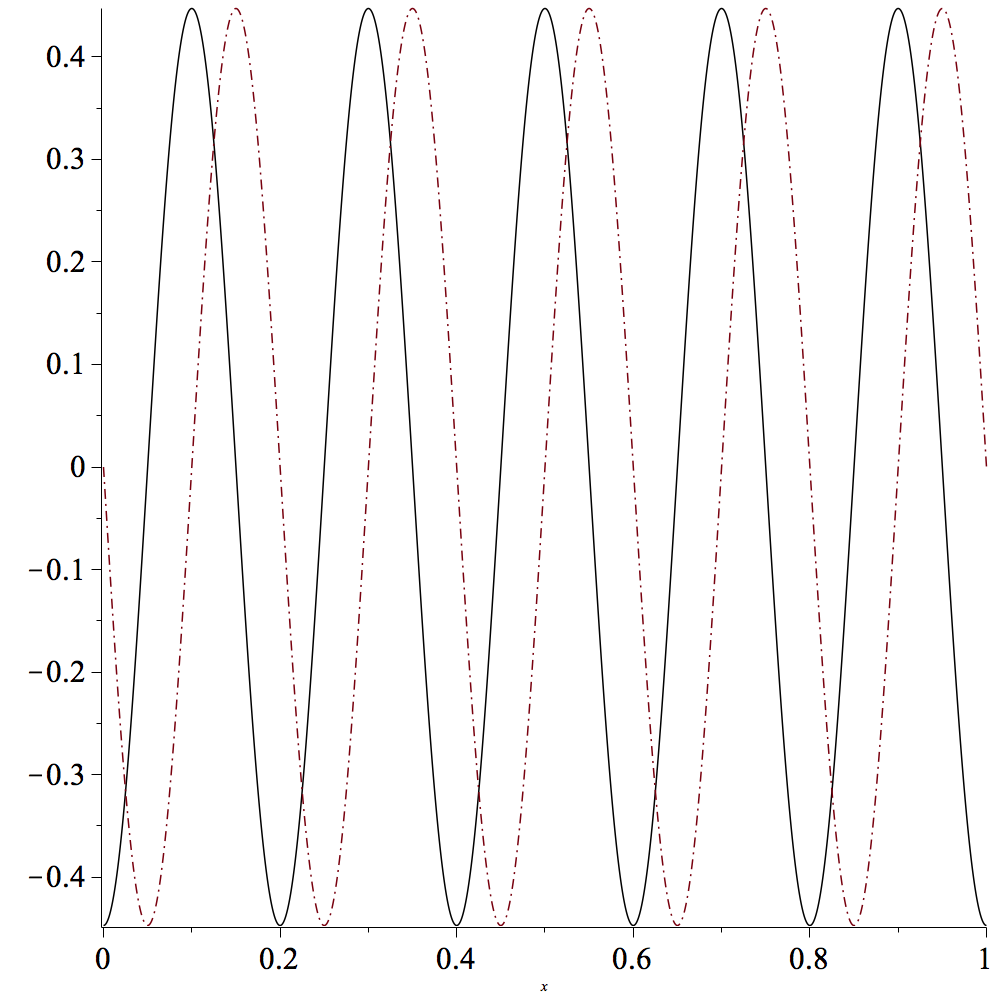}} &
\subfloat[$\alpha=0,~p=7$]{\includegraphics[width=.27\linewidth]{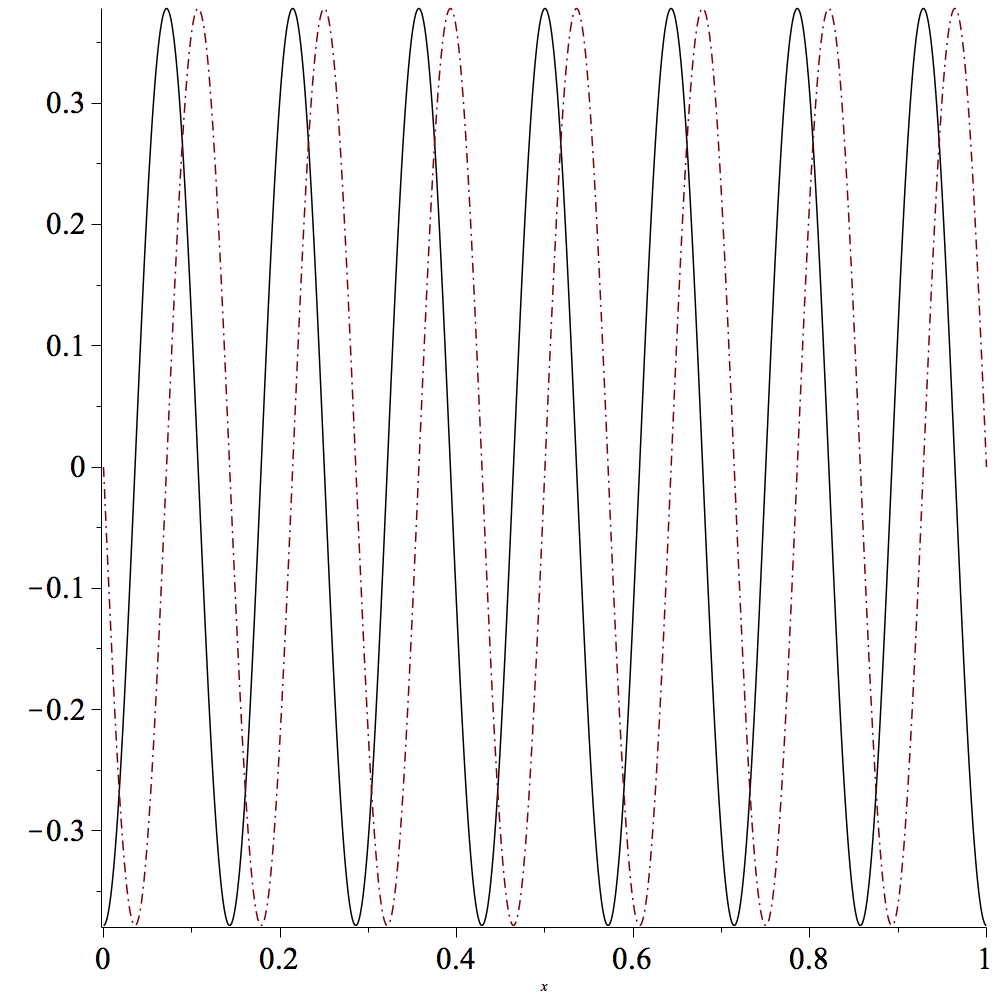}} \\
\subfloat[$\alpha=0.5,~p=1$]{\includegraphics[width=.27\linewidth]{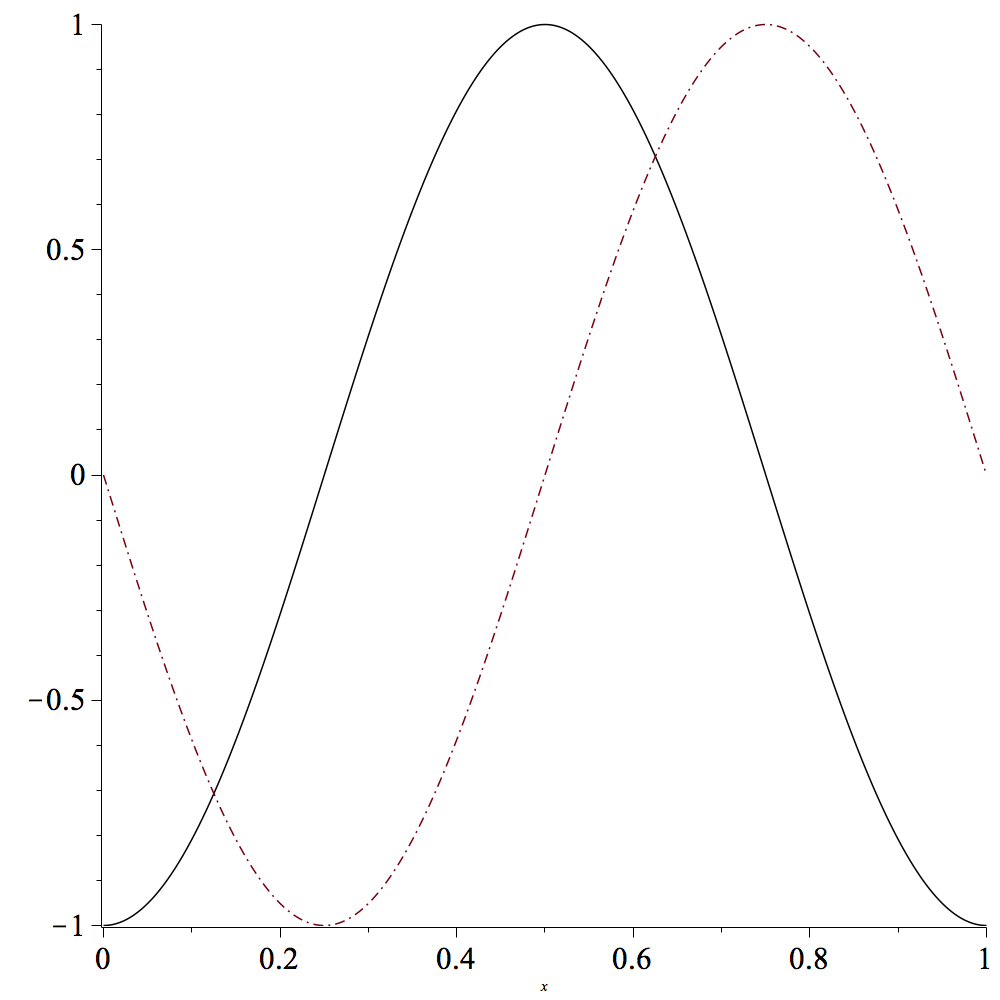}} &
\subfloat[$\alpha=0.5,~p=5$]{\includegraphics[width=.27\linewidth]{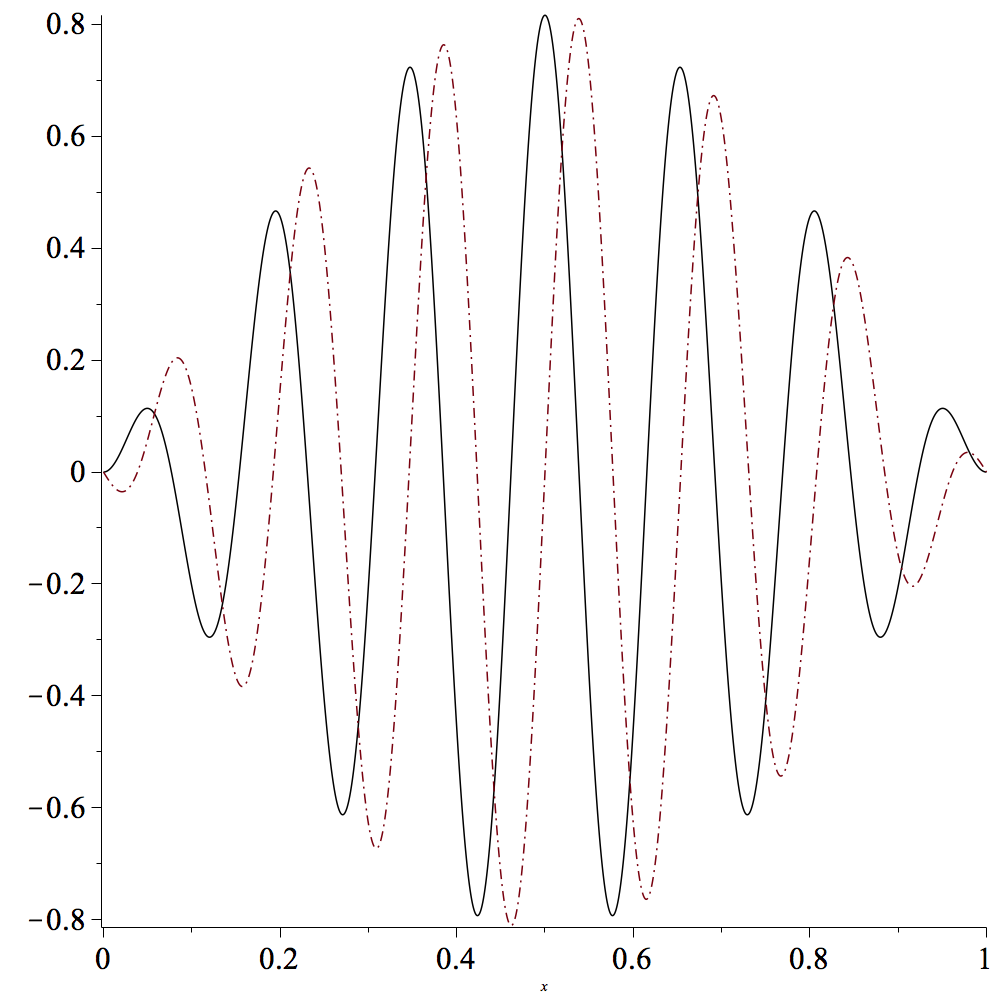}} &
\subfloat[$\alpha=0.5,~p=7$]{\includegraphics[width=.27\linewidth]{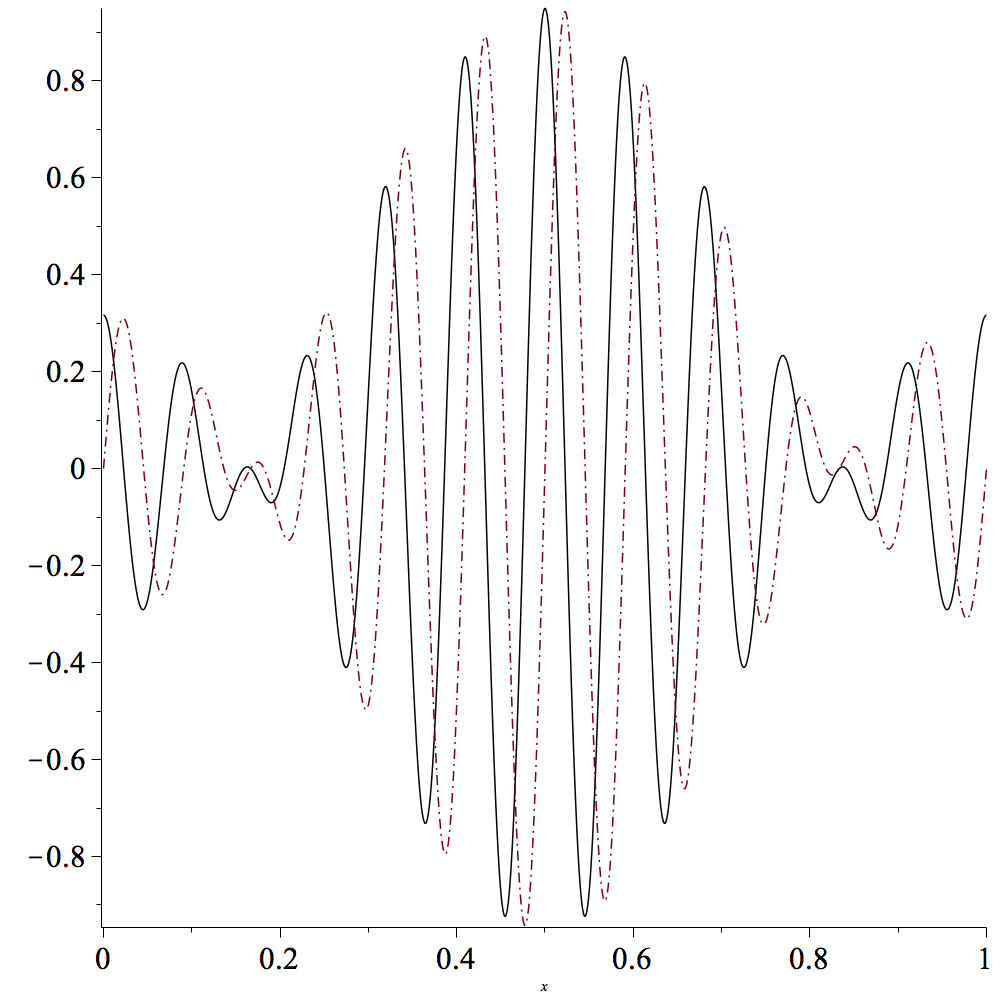}} \\
\subfloat[$\alpha=1,~p=1$]{\includegraphics[width=.27\linewidth]{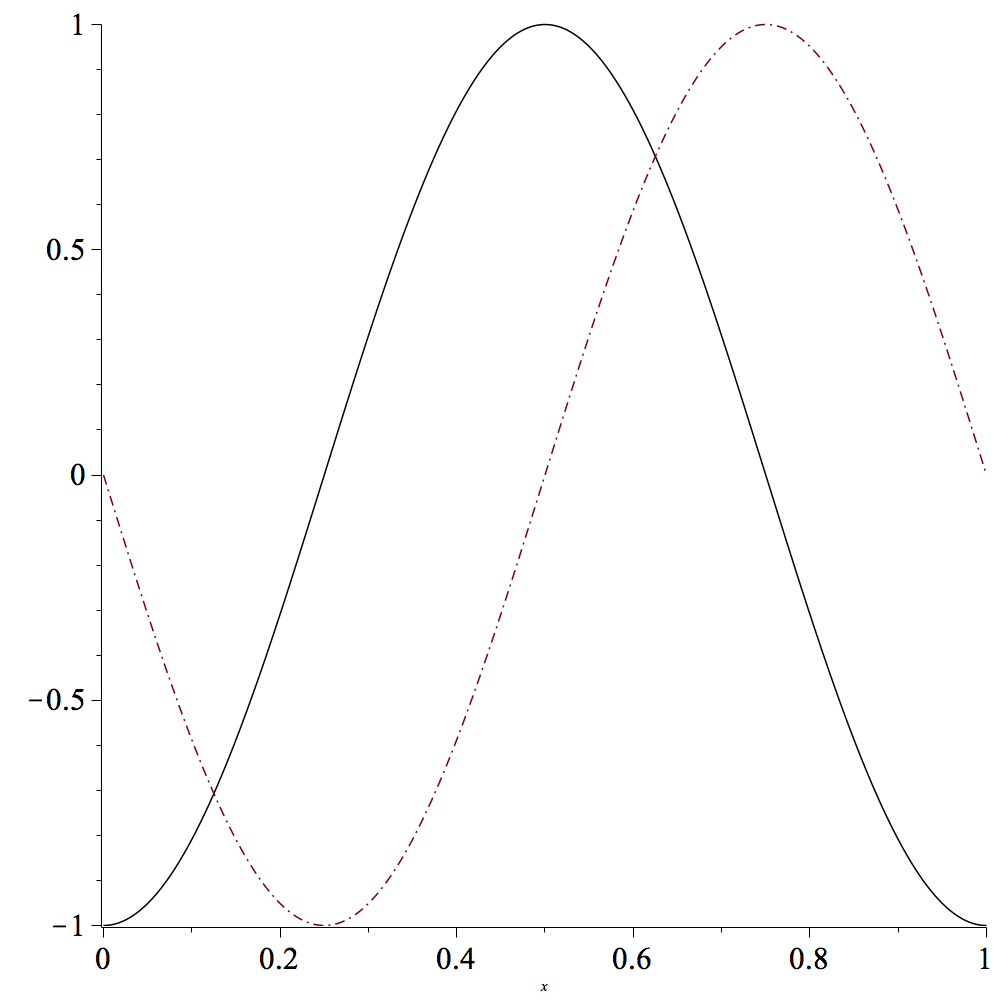}}  &
\subfloat[$\alpha=1,~p=5$]{\includegraphics[width=.27\linewidth]{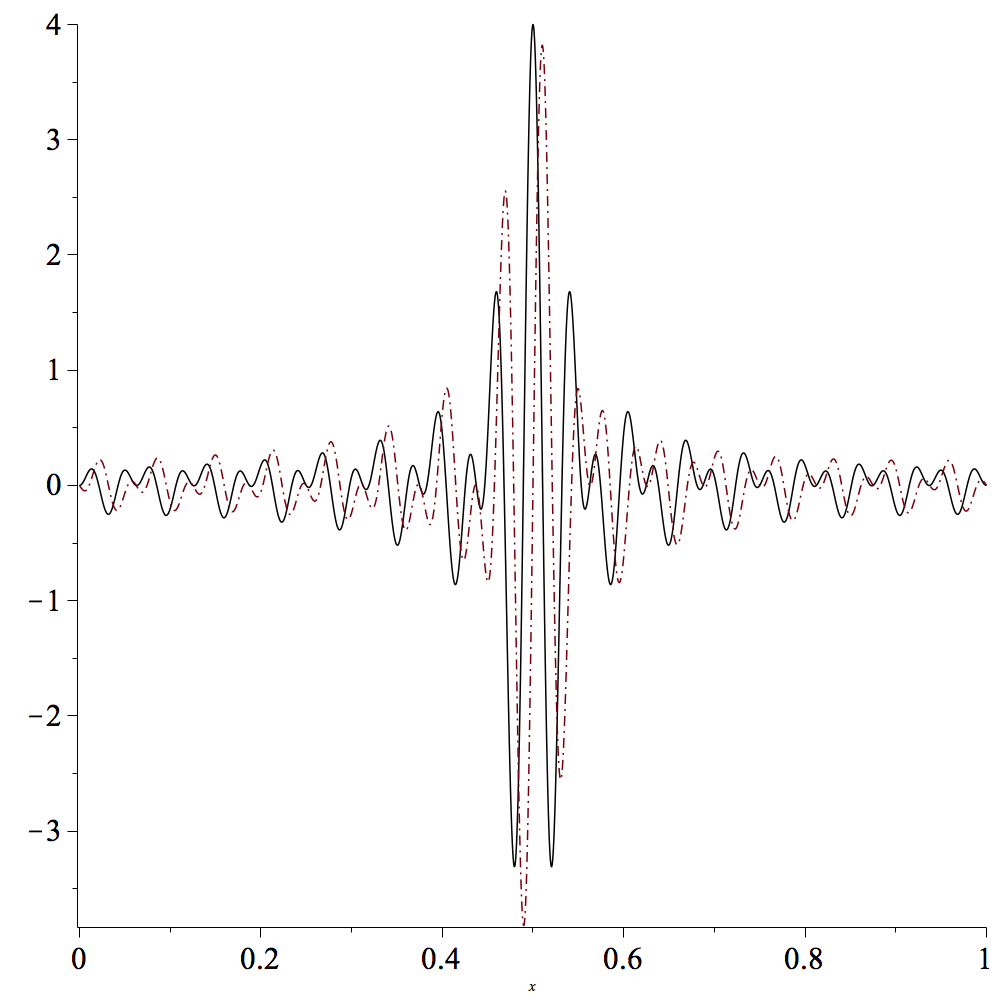}}  &
\subfloat[$\alpha=1,~p=7$]{\includegraphics[width=.27\linewidth]{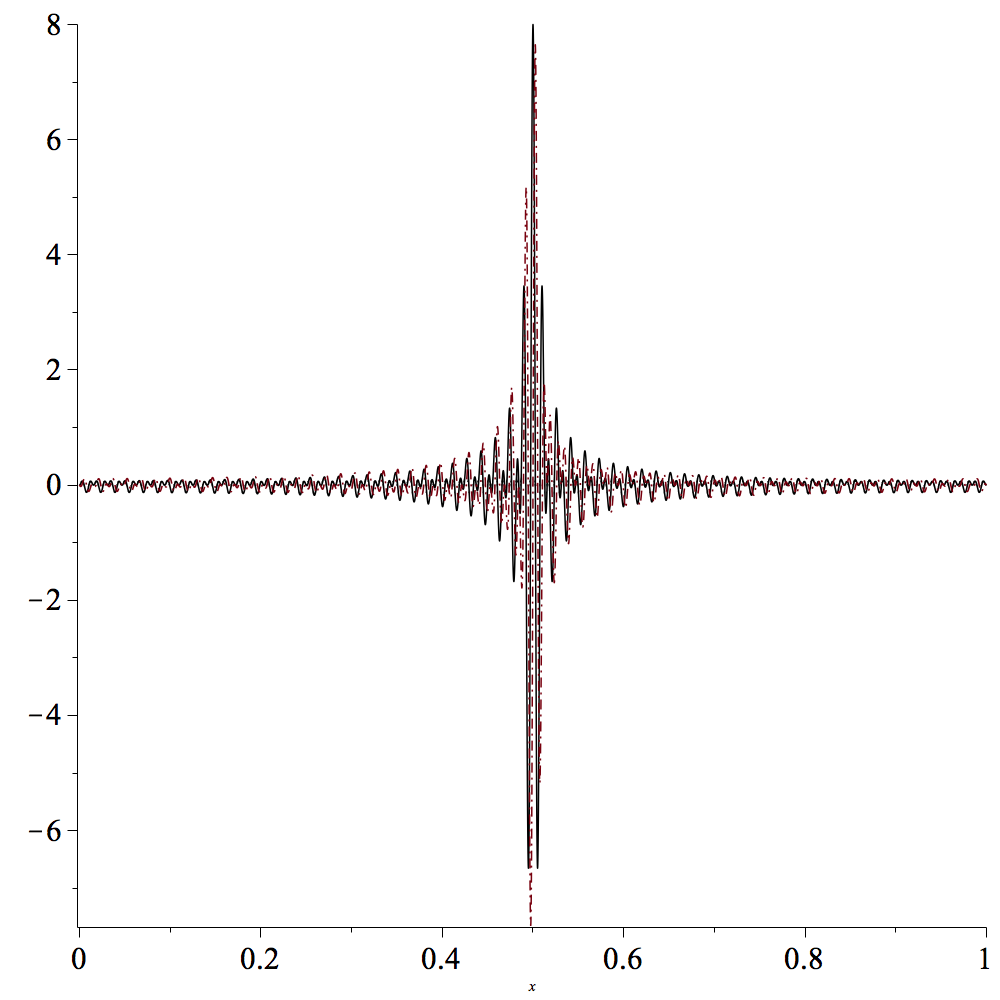}}  
\end{tabular}
\caption{Elements of the DOST bases. Notice that $\alpha=0$ is the classical Fourier basis which has no localization in time. As $\alpha$ or $p$ increases we gain time localization. The black line and the dotted-red one  represents the real and imaginary part respectively.}
\label{F: Dost_elements}		
\end{figure}

 For $p\in \Z_-$ we define
 \begin{align*}
    \Bpt(t)= \overline{B_{-p,\tau}}, \quad \tau=0, \ldots, \ba{-p}-1.
 \end{align*}
 Notice that if $\alpha=0$ then $\beta^0(p)=1$, for all $p$,
 therefore $\tau$ is always zero, hence
 \[
    B^0_{p,0}(t)= e^{\pii p t}.
 \]
 That is $\bigcup_{p} B^0_{p,0}$ is the ordinary Fourier basis of $\Lloi$.
 
 If $\alpha=1$, then $i_{1;0}=0$, $s_{1;0}=1$ and
 then $i_{1;p}=2^{p-1}$, $s_{1;p}=2^p$, $\beta_1(p)=2^{p-1}$ for $p\geq1$. Therefore
 \[
    B^1_{0,0}(t)=1
 \]
 while
 \[
    B^1_{p,\tau}(t)=\frac{1}{\sqrt{2^{p-1}} }\sum_{\eta=2^{p-1}}^{2^p-1} 
    e^{\pii \eta \pt{t- \frac{\tau}{2^{p-1}}}}, \quad \tau=0, \ldots, 2^{p-1}-1.
 \]
 That is $ \bigcup_{p, \tau} B^{1}_{p,\tau}$ is the so called DOST basis, see 
 \cite{Battisti2015}, \cite{ST07}.
 
 For now on, we restrict to positive integers,
 all the results hold true also for negative integers via simple arguments.
 
 Clearly the above partitioning can be considered on $\R$ instead of $\Z$. 
Let us consider the partitioning of the real line.
 Notice that for $\alpha=0$, trivially 
 \begin{align*}
  \beta^0(p)=1 = \eta^0, \quad \forall \eta\in I_{0;p}. 
 \end{align*}
 If $\alpha \in (0,1]$ then $\bap=\lfloor\iap^\alpha\rfloor$ and
 for $\eta \in \Iap=[\iap, \iap+\lfloor\iap^{\alpha}\rfloor)$ we have, for $|p|$ large enough,
  \begin{align*}
   \frac{1}{2^{\alpha+1} }
   \leq
   \frac{\lfloor\iap^{\alpha}\rfloor}{\pt{2 \iap}^\alpha }
   \leq 
    \frac{\lfloor\iap^{\alpha}\rfloor} {\pt{\eta}^\alpha}
   \leq \frac{\iap^\alpha}{\iap^\alpha}=1.             
 \end{align*}
 Using Fornasier's notation, see \cite{Fornasier2007157},  we can write
 \[
    \abs{\Iap} \asymp |\eta|^\alpha, \quad \eta\in \Iap.
 \]
That is the above partitioning is an $\alpha$-covering. 
\begin{rem}
  Notice that the $\alpha$-partitioning introduced above is well defined for all $\alpha\in [0,1]$.
  The case $\alpha=1$ is not defined as a limit case, as in the usual analysis
  of $\alpha$-modulation spaces, see \cite{Feichtinger:2006rt,Fornasier2007157}. The key point is that we use an
  iterative scheme, instead of  a definition involving the function $p^{\frac{\alpha}{1-\alpha}}$.
  In this way, we can get rid of the singularity which arises at $\alpha=1$. At $\alpha=1$ the growth 
  of $\beta(p)$ with respect to $p$ is exponential while for $\alpha\in [0,1)$ is polynomial, see
 Figure~\ref{F:overl}.
  Using the iterative definition of $\beta(p)$ this fact causes no problems. 
\end{rem}
\begin{figure}
                \includegraphics[width=.7\textwidth]{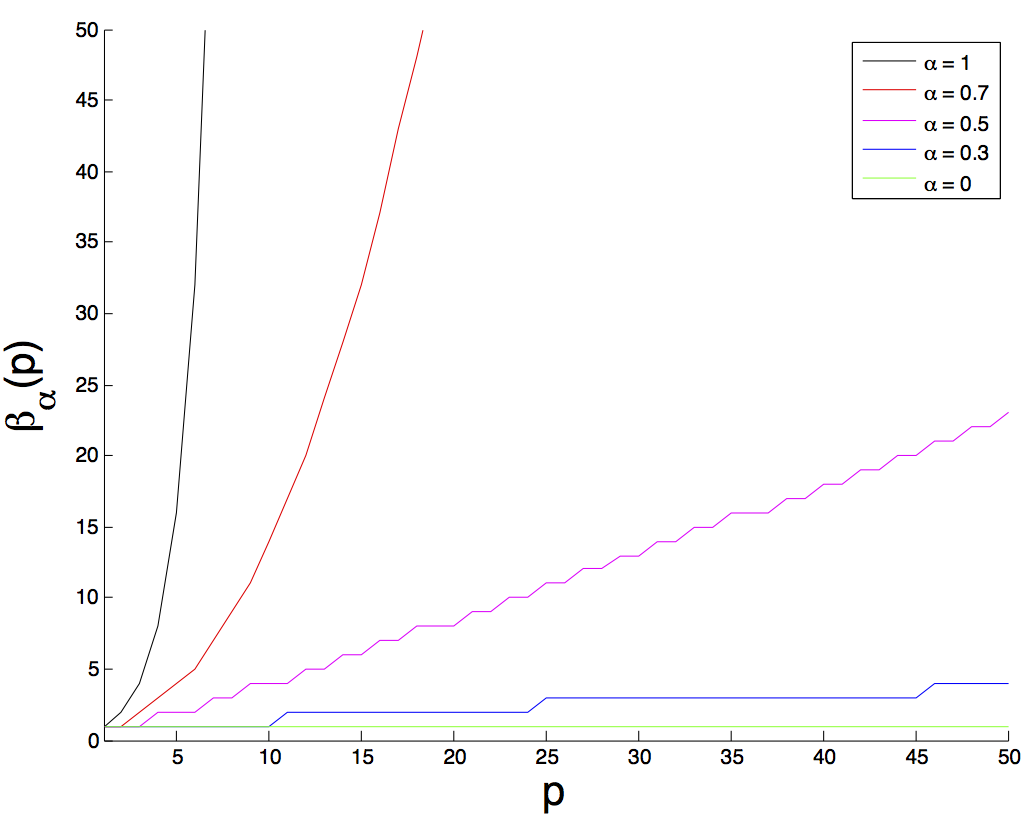}
                \caption{Growth of $\bap$ with the respect of the parameter $\alpha$.}\label{F:overl}
                \end{figure}
\begin{thm}
The functions
\[
 \Bpt\pt{\cdot}, \quad p\in \Z, \tau=0, \ldots, \bap-1
\]
form an orthonormal basis of $\Lloi$.
\end{thm}
\begin{proof}
 Notice that for $\alpha=0$ and $\alpha=1$ the Theorem
 holds true in view of well known properties of Fourier basis and 
 of results proven in \cite{Battisti2015}.
 
 The proof follows closely the argument in \cite{Battisti2015}.
 \\
 \noindent \textbf{Step 1} \emph{$\norm{\Bpt}_{\Lloi}=1$.}\medskip\\
 Since $\ptg{e^{\pii k t}}_{k\in \Z}$ is an orthonormal basis of $\Lloi$ we can write 
 \begin{align*}
   \pt{\Bpt, \Bpt}_{\Lloi}&= \frac{1}{\bap}\sum_{\eta=\iap}^{\sap-1} \sum_{\eta'=\iap}^{\sap-1} \int
   e^{\pii \eta\pt{t-\frac{\tau}{\bap}} } e^{-\pii \eta'\pt{t-\frac{\tau}{\bap}} } dt\\
                          &= \frac{1}{\bap} \sum_{\eta=\iap}^{\sap-1} \sum_{\eta'=\iap}^{\sap-1} \delta_{0}\pt{\eta-\eta'}\\
                          &= \frac{1}{\bap} \sum_{\eta=\iap}^{\sap-1}1\\
                          &=1.
 \end{align*}
 \medskip\\
 \noindent \textbf{Step 2} \emph{$\pt{\Bpt,B^\alpha_{p',\tau'}}_{\Lloi}=\delta_0\pt{p-p'}\delta_0
 \pt{\tau-\tau' }$.}
 \medskip\\
 \noindent If $p=p'$ and $\tau=\tau'$ the first step implies the assertion. 
  Moreover, if $p\neq p'$ then  $\Ia p$ and $\Ia {p'}$ are disjoint, thus
  \begin{align*}
      \pt{\Bpt,B^\alpha_{p',\tau'}}_{\Lloi}=0.
  \end{align*} 
 So, we can restrict to the case $p=p'$.
 Since $\ptg{e^{\pii k t}}_{k\in \Z}$ is an orthonormal basis of $\Lloi$, we obtain 
 \begin{align*}
   \pt{\Bpt,B^\alpha_{p,\tau'}}_{\Lloi}&=\frac{1}{\bap} \sum_{\eta=\iap}^{\sap-1}  e^{\pii\pt{\tau-\tau'}\frac{\eta}{\bap}}\\
				      &=\frac{1}{\bap} \sum_{j=0}^{\bap-1}  e^{\pii\pt{\tau-\tau'}\frac{\pt{\iap+j}}{\bap}}.
 \end{align*}
 Let us suppose $\tau-\tau'\neq 0$, then we can write
 \begin{align}
  \nonumber
  \pt{\Bpt,B^\alpha_{p,\tau'}}_{\Lloi}&=\frac{ e^{\pii \pt{\tau-\tau'} \frac{\iap}{\bap} }}{\bap} \sum_{j=0}^{\bap-1}  \pt{e^{\pii\pt{\tau-\tau'}\frac{1}{\bap}}}^j\\
   \label{eq:geo}
   &= \pt{\frac{ e^{\pii \pt{\tau-\tau'} \frac{\iap}{\bap} }}{\bap}}\pt{ \frac{1- e^{\pii \pt{\tau-\tau'} \frac{\bap}{\bap}}}{1-e^{\pii \pt{\tau-\tau'} \frac{1}{\bap}}}}\\
   \nonumber
   &=0.
 \end{align}
 In equation \eqref{eq:geo} we used well known properties of geometric series and the fact that $\frac{\tau-\tau'}{\bap}$
 is not an integer number and therefore $e^{\pii \pt{\tau-\tau'} /\bap} \neq 1$.
 \medskip\\
 \noindent \textbf{Step 3} \emph{The functions $\Bpt$ generate $\Lloi$.}
 \medskip\\
 Notice that 
 \[
   \bigcup_{\tau=0}^{\bap-1} \Bpt \subseteq \Span\ptg{e^{\pii k t}}_{k\in [\iap, \sap-1]}.  
 \]
 Therefore, to prove that the functions $\Bpt$ are a basis of $\Lloi$ it is sufficient to check
 that $\bigcup_{\tau=0}^{\bap-1} \Bpt$ are linear independent.
 That is, if
 \begin{align}
   \label{eq:base}
   \sum_{\tau=0}^{\bap-1} \alpha_\tau \Bpt=0 \Rightarrow \alpha_\tau=0, \quad \tau=0, \ldots, \bap-1. 
 \end{align}
 Let us consider the projection of \eqref{eq:base} into the Fourier basis. We obtain the system of equation
 \begin{align}
   \label{eq:sistalpha}
   \sum_{\tau=0}^{\bap-1} \alpha_\tau e^{-\pii \tau \frac{\iap+j}{\bap}} = 0, \quad j=0, \ldots, \bap-1.
 \end{align}
 We can rewrite the above equation as a linear system
 \begin{align}
    \label{eq:vander}
    \left(
          \begin{array}{cccc}
           1 & e^{-\pii \frac{\iap}{\bap}}& \ldots & e^{-\pii (\bap-1) \frac{\iap}{\bap} }\\
           1 & e^{-\pii \frac{\iap+1}{\bap}}& \ldots & e^{-\pii (\bap-1) \frac{\iap+1}{\bap} }\\
           \vdots & \vdots & \ddots &\vdots\\
           1 & e^{-\pii \frac{\iap+\bap-1}{\bap}}& \ldots & e^{-\pii  (\bap-1)\frac{\iap+\bap-1}{\bap} }
          \end{array}
          \right) \cdot 
          \left(
           \begin{array}{c}
             \alpha_0\\
             \alpha_1\\
             \vdots\\
             \alpha_{\bap-1}
           \end{array} 
          \right)
          =
          \left(
           \begin{array}{c}
             0\\
             0\\
             \vdots\\
             0
           \end{array} 
          \right).
 \end{align}
 The square matrix in \eqref{eq:vander} is a Vandermonde matrix. Since
 the entries $\pt{e^{-\pii (\iap+j)/\bap}}_{j=0}^{\bap-1}$ are all distinct
 the determinant of the matrix is zero and therefore the system in \eqref{eq:vander}
 has only the null solution, that is in \eqref{eq:sistalpha} and \eqref{eq:base},
 $\alpha_\tau$ mus vanish for all $\tau=0, \ldots, \bap-1$.
 \end{proof}
 \subsection{Localization properties of the functions $\Bpt$}
 Now, we investigate the time-frequency localization properties of the functions $\Bpt$.
 They clearly have compact support in the frequency domain and the support is
 precisely $\Iap$. Therefore, they cannot have compact support in the time domain. Nevertheless,
 it is possible to determine localization properties in the spirit
 of Donoho-Stark Theorem, \cite{DS89}. 
 \begin{Prop}
   For each $\alpha\in [0, 1]$ and $p, \tau$, the following inequality holds
   \[
     \pt{\int_{\frac{\tau}{\bap}-\frac{1}{2 \bap} }^{\frac{\tau}{\bap}+\frac{1}{2 \bap}} |\Bpt|^2 dt}^2\geq 0.85.
   \]
   that is the $L^2$-norm is concentrated in the interval
   \[
     \ptq{\frac{\tau}{\bap}-\frac{1}{2 \bap}, \frac{\tau}{\bap}+\frac{1}{2 \bap}}
   \]
   if $\tau=0$, the interval must be considered as an interval in the circle:
   \[
     \ptq{0, \frac{1}{2 \bap}}\cup \ptq{1-\frac{1}{2 \bap}, 1}.
   \]

 \end{Prop}
\begin{proof}
  The proof is based on Taylor expansion and Gauss summation formula. 
  In \cite{Battisti2015} the Property is proven for the case $\alpha=1$, actually the same proof works also for
  general $\alpha\in [0,1]$.
\end{proof}

%\end{document}

%%%%%%%%%%%%%%%%%%%%%%%%%%%%%%%%%%%%%%%%%%%%%%%%%%%%%%%%%%
%%%%%%%%%%%%%%PARTE SUL FRAME E PAINLESS%%%%%%%%%%%%%%%%%%
%%%%%%%%%%%%%%%%%%%%%%%%%%%%%%%%%%%%%%%%%%%%%%%%%%%%%%%%%%
\section{DOST-wave packet frame, dimension $d=1$.}\label{sec:dim1}
%%%
In order to give a more flexible structure to the DOST basis, we generalize it to a redundant and non-orthogonal system and show that this leads to a $L^2$-frame.
First, we address the $1$-dimensional case in full generality taking into account the phase-space tiling dependent from the parameter $\alpha$ defined in the first part of the paper.%%
\subsection{Frame definition}         
The idea is to extend by periodicity the DOST basis and then localize them using a
suitable window function.
\begin{Def}
Consider
\begin{equation} \label{eq:DOST_sys}
    \vpt(t) =  \frac{1}{\sqrt{\ba{p}}}
                      \sum_{\eta\in\mu\Iap }
                      e^{\pii \eta \pt{t- k \frac{\nu}{\ba{p}}}}
                      \varphi\pt{t- k \frac{\nu}{\ba{p}}}, \qquad
                      (p,k) \in \Z^2,
\end{equation}
where $\Iap$  is defined in Section \ref{S:alfamod}. Then, we define the $\alpha$-DOST system 
\[
D^{\alpha}\pt{\mu,\nu, \phi} = \ptg{\vpt(t)}_{p,k}, \quad (p,k) \in \Z^2,
\]
see Figure~\ref{F: FramePic} for the plot of the frame elements. We did not plot the case $\alpha =0$ which is the standard Gabor frame element.
\end{Def}
\begin{rem}
Notice that the Fourier transform of \eqref{eq:DOST_sys} simplifies into 
\begin{align*}
    \widehat{\varphi}^{\alpha}_{p,k}(\omega)
    		& = \frac{1}{\sqrt{\ba{p}}}
                     e^{-\pii \omega k\frac{\nu}{\ba{p}}}
                      	\sum_{\eta\in\mu\Iap }
                      		\hat{\varphi}(\omega- \eta)
                				\qquad  (p,k) \in \Z^2.
\end{align*}
This kind of system of functions has affinities with the well known family of Non-Stationary Gabor frame, see \cite{MR2854065, MR3211110,MR3337493}.\\  The function 
\[
  \Gapm(\omega)=\sum_{\eta\in\mu\Iap }
                      		\hat{\varphi}(\omega- \eta)
\]
mimic a bell function of the set $\Iap$ as shown in Figure~\ref{F:FrWin}.
\end{rem}

\begin{rem}
Sometimes, when we want to highlight the dependence of the element $\Bpg(x)$ on the frame parameters, we use the notation $\Bapkmn$.	
\end{rem}

\begin{rem}
As we pointed out is Section \ref{S:alfamod}, when $\alpha =0$ the DOST basis reduces to the Fourier one. As one may expect, when we analyze the same case for the $\alpha$-DOST system above, we end up with a Gabor frame. Precisely:
\[
D^{0}\pt{\mu,\nu, \phi} = \mathcal{G}(\mu,\nu, \phi) = \ptg{T_{\nu k}M_{\mu p}\phi, \quad (p,k) \in \Z^2}.
\]
Indeed, recalling that $\bap =1$ when $\alpha=0$, $\iap = p$ we thus obtain
\[
\varphi^{0}_{p,k}(x) = 
                      \sum_{\eta = \mu p }
                      e^{\pii \eta \pt{x-\nu k}}
                      \varphi(x-\nu k) = T_{\nu k}M_{\mu p}\varphi.
\]
\end{rem}
\begin{figure}
\begin{tabular}{cc}
\subfloat[{$\alpha=0.5,~p=5$}]{\includegraphics[width=.33\linewidth]{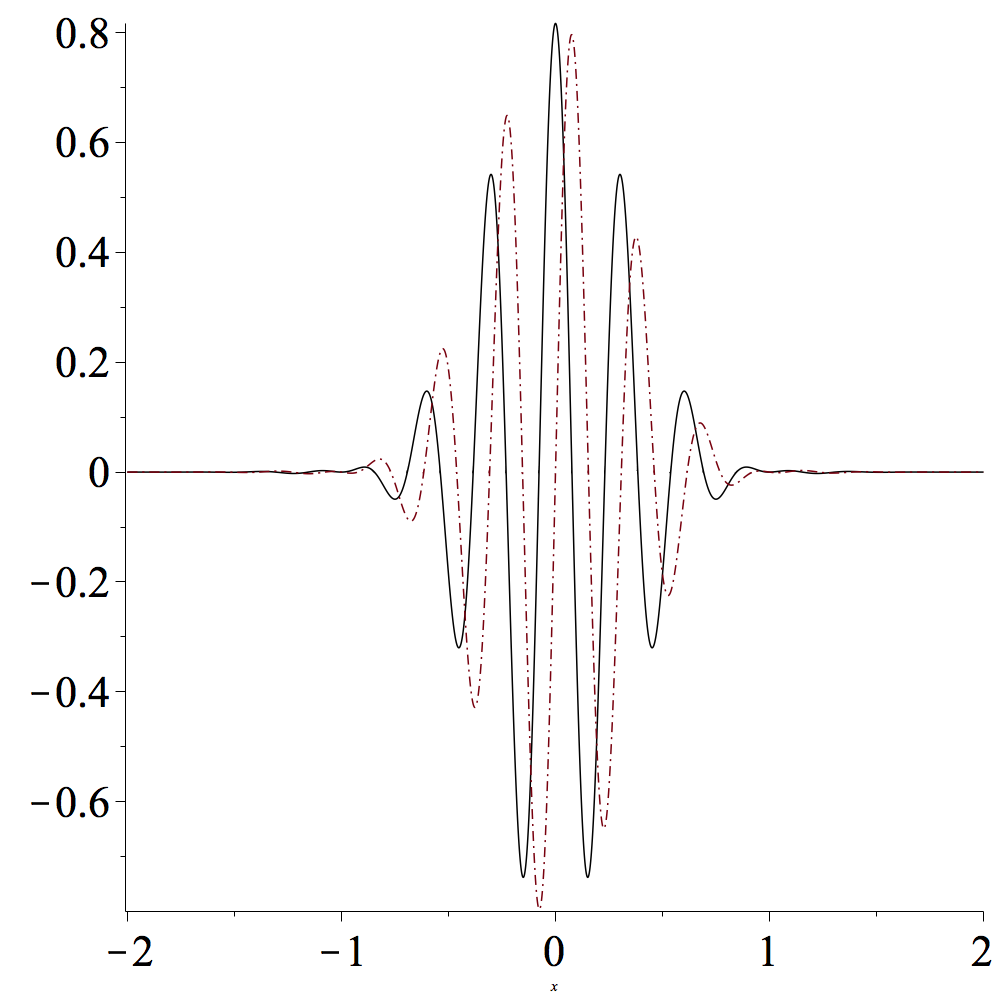}} &
\subfloat[{$\alpha=0.5,~p=7$}]{\includegraphics[width=.33\linewidth]{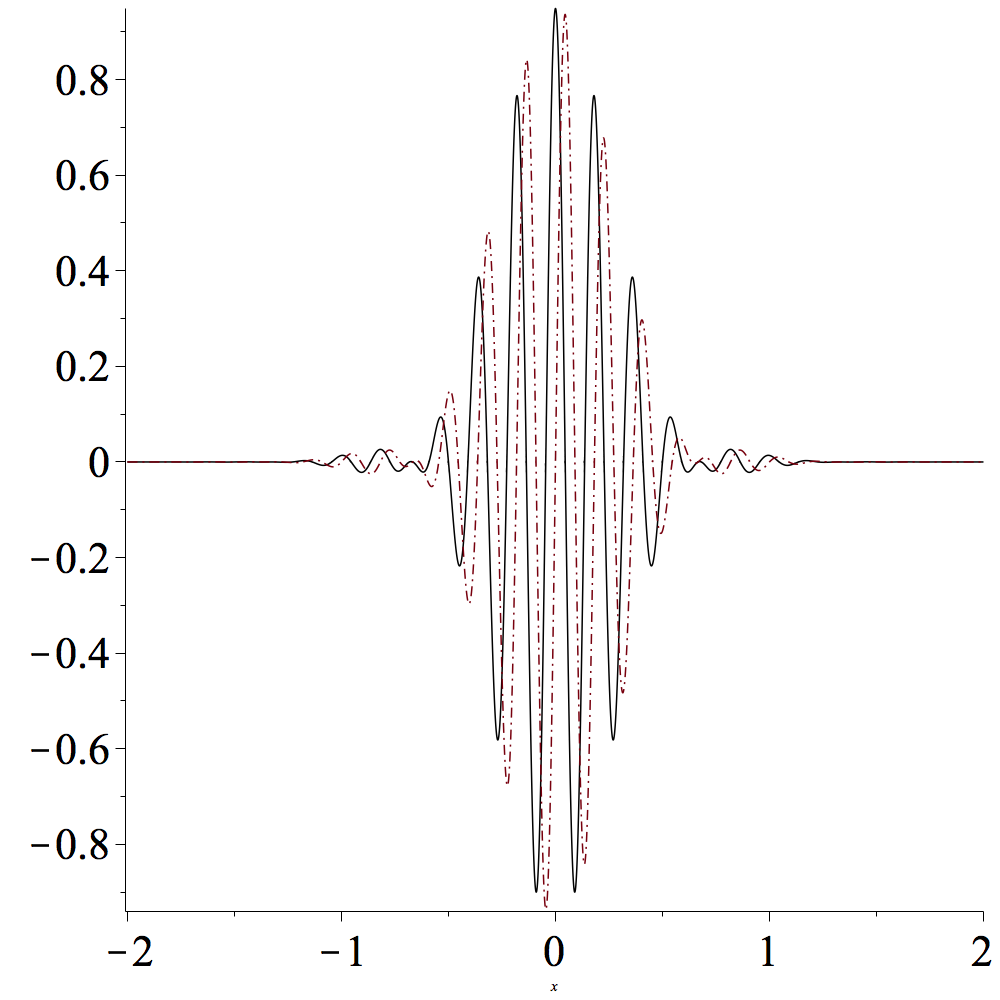}} \\
\subfloat[{$\alpha=1,~p=5$}]{\includegraphics[width=.33\linewidth]{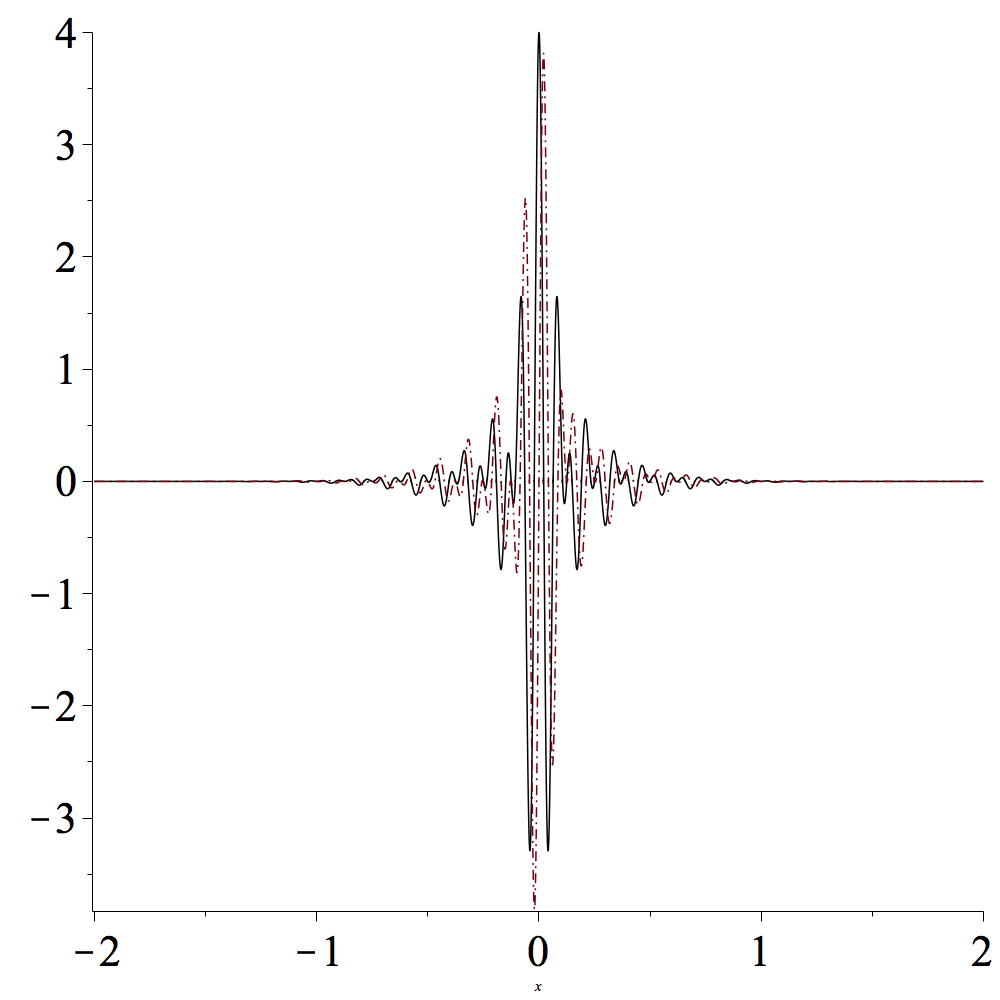}} &
\subfloat[{$\alpha=1,~p=7$}]{\includegraphics[width=.33\linewidth]{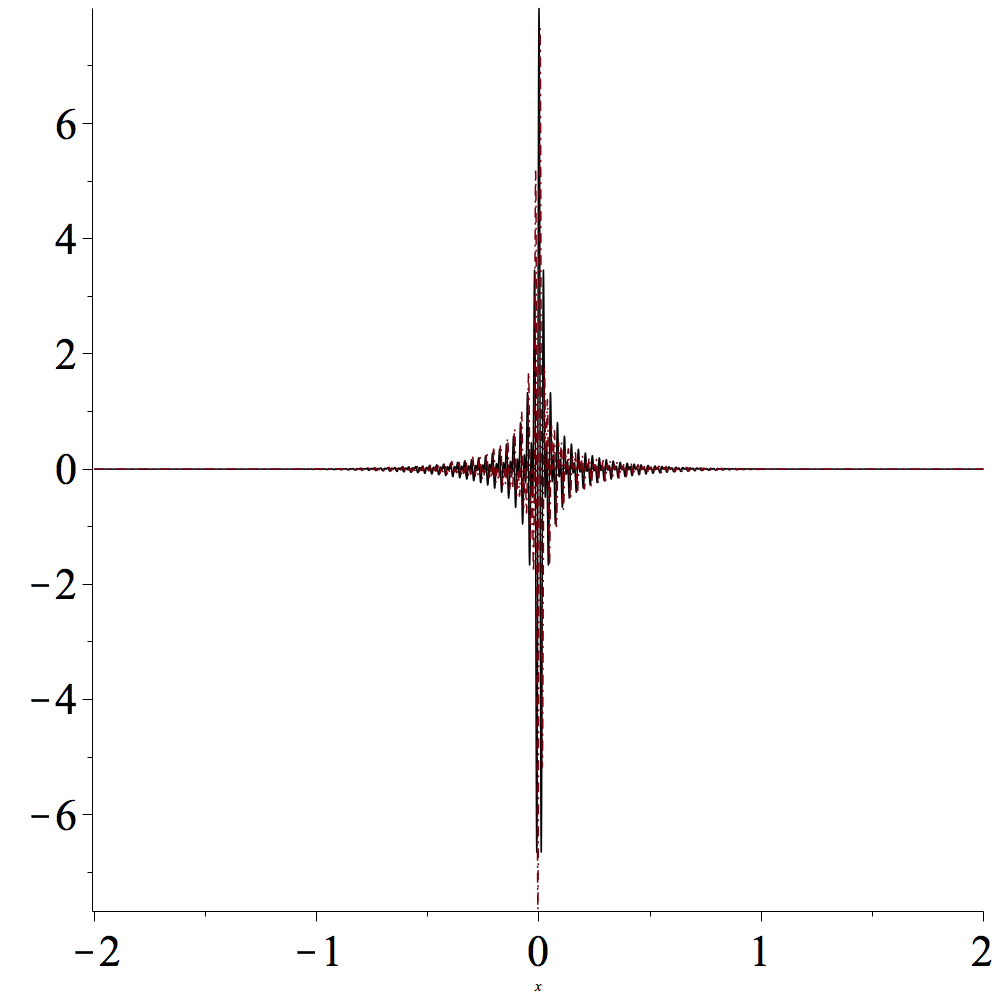}}
\end{tabular}
\caption{Plot of the elements of the $\alpha$-DOST frame in time, with a Gaussian function as a window and $\mu=1/2$. As we did for the bases, we observe that the localization increases when $\alpha$ or $p$ grows. The black line and the dotted-red one  represents the real and imaginary part respectively.}
\label{F: FramePic}		
\end{figure}
\vspace{2em}
\begin{figure}
\begin{tabular}{cc}
\subfloat[$\alpha=0.5,~p=5$]{\includegraphics[width=.33\linewidth]{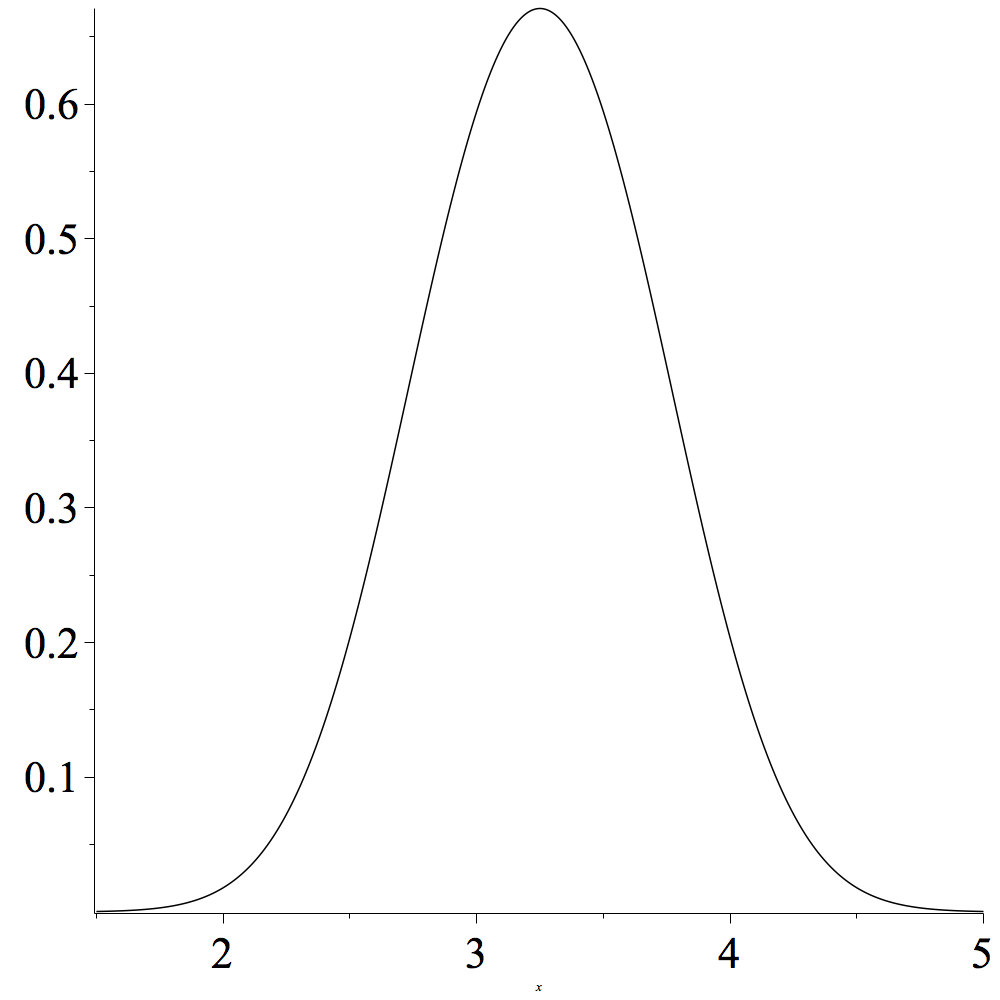}} &
\subfloat[$\alpha=0.5,~p=7$]{\includegraphics[width=.33\linewidth]{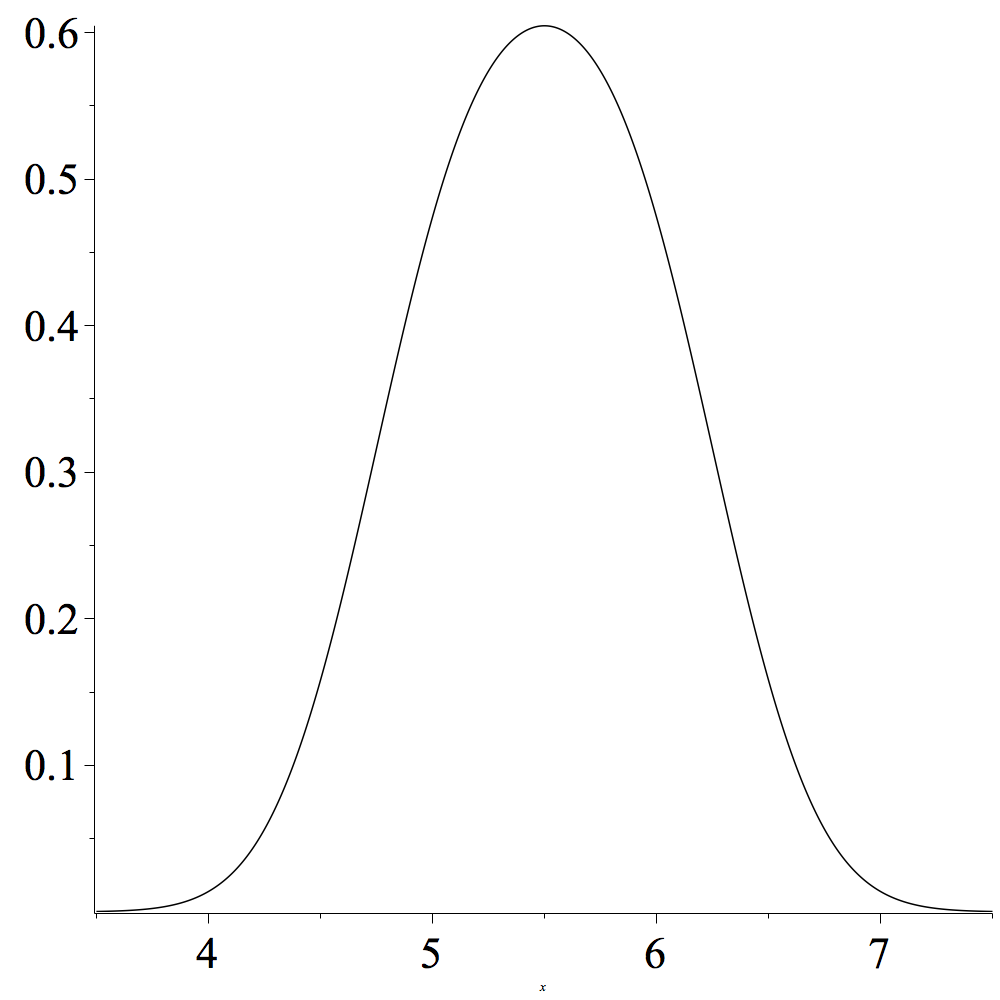}} \\
\subfloat[$\alpha=1~p=5$]{\includegraphics[width=.33\linewidth]{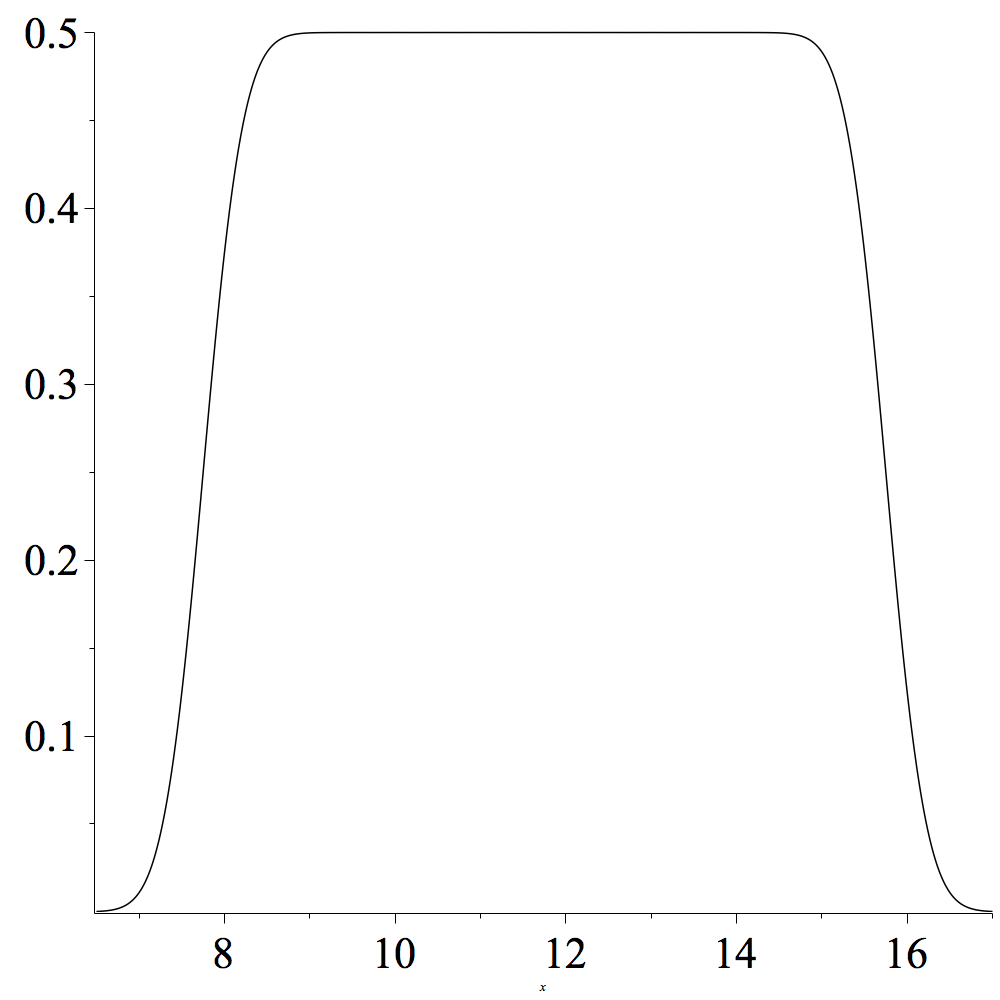}} &
\subfloat[$\alpha=1,~p=7$]{\includegraphics[width=.33\linewidth]{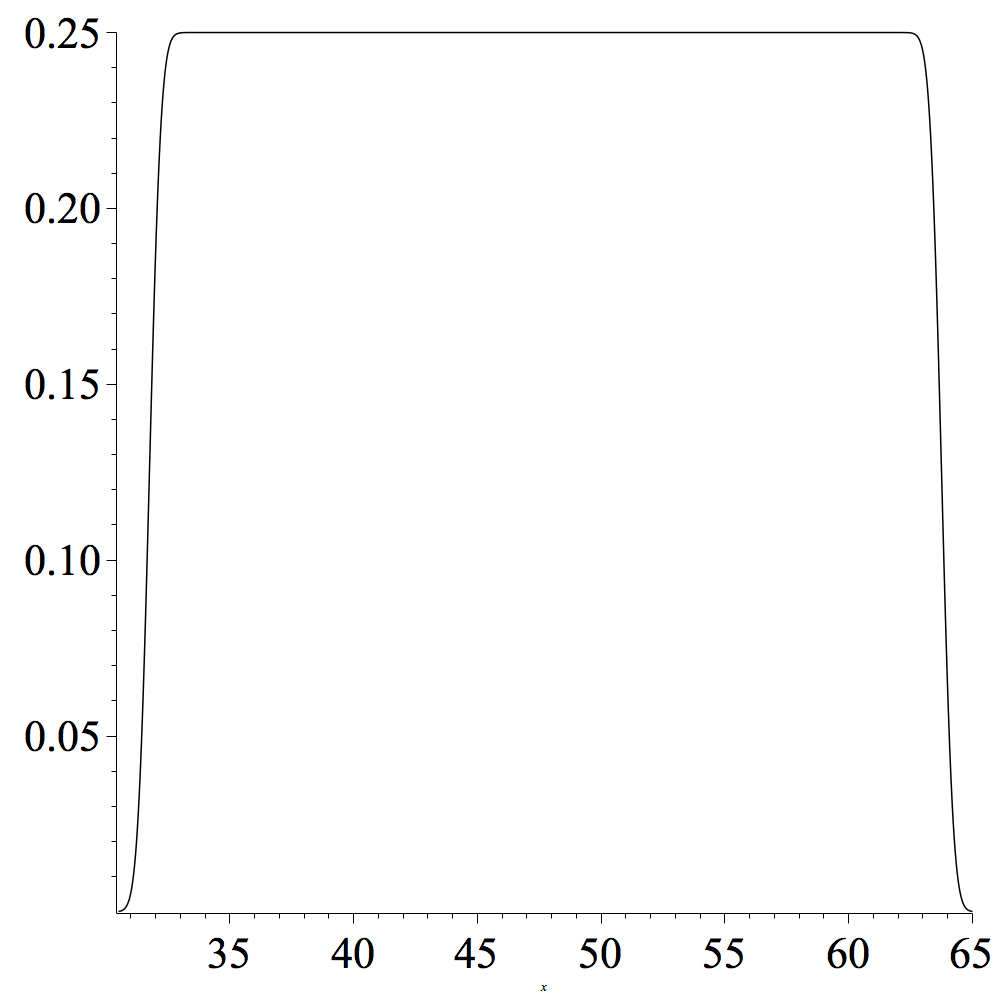}}
\end{tabular}
\caption{Plot of the elements of the DOST-frame in frequency, with a Gaussian function as a window and $\mu=1/2, \nu =1/5$. Notice that as $p$ grows, the window translates and approximates a bell function of the set $\Iap$.}
\label{F:FrWin}		
\end{figure}
We investigate the condition under which the family $D^{\alpha}\pt{\mu, \nu, \varphi}$ is a frame of $\Ll$. 
%%%%%%%%%%%%%%%%%%%%%%%%%%%%%%%%%%%%%%%%%%%%%%%%%%%%%%%%%%%%%%%%%%%%%%%%%%%%%%%%%%%%%%%%%%%%%%%%
%%%%%%%%%%%%%%%%%%%%%%%WALNUT LIKE FORMULA%%%%%%%%%%%%%%%%%%%%%%%%%%%%%%%%%%%%%%%%%%%%%%%%%%%%%%
%%%%%%%%%%%%%%%%%%%%%%%%%%%%%%%%%%%%%%%%%%%%%%%%%%%%%%%%%%%%%%%%%%%%%%%%%%%%%%%%%%%%%%%%%%%%%%%%
\subsection{Walnut-like Formula}
\label{subsec:Walnut}
We retrieve a Walnut-like formula for our DOST system in the spirit of Gabor analysis. This representation is a very useful tool to prove the frame property. Moreover, when $\alpha = 0$, we recover the Walnut representation for Gabor frames.\\
First, we recall the Poisson formula, \cite[Proposition 1.4.2]{GR01}.
\begin{lem}
Suppose that for some $\epsilon>0$ and $C>0$ we have $|\phi(x)| \leq C(1+|x|)^{-d-\epsilon}$ and $|\hat{\phi}(\omega)|\leq C(1+|\omega|)^{-d-\epsilon}$. Then
\begin{equation}
\label{eq:poiss}	
\gamma^d \sum_{n\in \Z^d} \phi(x+\gamma n) 
			= \sum_{n\in \Z^d}\hat{\phi}\pt{ 
											\frac{n}{\gamma}
										}
							  e^{\pii nx/\gamma},			
\end{equation}
for $\gamma>0$, where $d$ is the dimension of the base space.
\end{lem}
We can state the main result of this section.
\begin{lem}\label{L:2}
Let $\phi,\psi \in L^2(\R)$, $\hat{f}\in C^{\infty}_c(\R)$, and
\begin{align}
  \label{eq:direct}
  \pt{S^{\alpha;\nu;\mu}_{\varphi,\psi} f} (t) = \sum_{p,k}\langle f,\Bapkmn \rangle \PBapkmn(t), \quad t\in \R, \;
  (p,k) \in \Z^2.  
\end{align}
Then 
\begin{equation}
  \label{eq:Walnut}
  \F_{t\to \omega} \pt{\PFamnfg f} (\omega)	=\nu^{-1} \sum_{p\in\Z}\sum_{k\in \Z}T_{k\frac{\bap}{\nu}}\pt{\hat{f}\pt{\omega} }\overline{\Gapm}\pt{\omega-k\frac{\ba{p}}{\nu}}\PGapm(\omega), 
\end{equation}
where 
\[
\PGapm(\omega) = \sum_{\eta\in\mu\Iap }
                      		\hat{\psi}(\omega- \eta).
\]
\end{lem}
%%%%%%%%%%%%%%%%%%%%%%%%%%%%%%%%%%%%%%%%%%%%%%%%%%%%%%%%%%%%%%%%%%%%%%%%%%%%%%%%%%%%%%%%%%%%%%%%%%%%%%%%%%
%%%%%%%%%DIMOSTRAZIONE%%%%%%%%%%%%%%%%%%%%%%%%%%%%%%%%%%%%%%%%%%%%%%%%%%%%%%%%%%%%%%%%%%%%%%%%%%%%%%%%%%%%%
%%%%%%%%%%%%%%%%%%%%%%%%%%%%%%%%%%%%%%%%%%%%%%%%%%%%%%%%%%%%%%%%%%%%%%%%%%%%%%%%%%%%%%%%%%%%%%%%%%%%%%%%%%
\begin{proof}
Take the Fourier transform of the frame operator and obtain
\begin{align*}
  \F_{t\to \omega} \pt{\PFamnfg f}(\omega) &{}=  \sum_{p,k\in\Z^2} \langle \hat{f},\widehat{\Bapkmn}\rangle \widehat{\PBapkmn}(\omega)
  \\
  &{}=  \sum_{p,k\in\Z^2} 
 			\langle\hat{f}(\cdot), 
 				\frac{1}{\sqrt{\bap}}e^{-\pii k \frac{\nu}{\bap} (\cdot)}\Gapm(\cdot)
 			\rangle
 				\frac{1}{\sqrt{\bap}}
 					e^{-\pii k \frac{\nu}{\bap}\omega}
 						\PGapm(\omega)
\\
 &{}= \sum_{p\in\Z}\pt{
 						\sum_{k\in \Z}
 							\langle \hat{f}(\cdot),\frac{1}{\sqrt{\bap}}
 								e^{\pii k \frac{\nu}{\bap}(\cdot)}\Gapm(\cdot)
 			                \rangle 
 			             \frac{1}{\sqrt{\ba{p}}}
 							e^{\pii k \frac{\nu}{\bap} \omega}
 					  }\PGapm(\omega) 
\\
&{}= \nu^{-1}\sum_{p\in\Z}\pt{
 						\frac{\nu}{{\bap}} \sum_{k\in \Z}
 							\langle \hat{f}(\cdot),
 								e^{\pii k \frac{\nu}{\bap}(\cdot)}\Gapm(\cdot)
 			                \rangle 
 							e^{\pii k \frac{\nu}{\bap} \omega }
 					  }\PGapm(\omega)  					  
\end{align*}
Then, \eqref{eq:poiss} with $\phi = \hat{f}\overline{\Gapm}$, $n = k$ and $\gamma = \frac{\bap}{\nu}$ yields
\begin{align*}
F_{x\to \omega}\pt{\Famnfg f}(\omega) 
	&{}= \nu^{-1}\sum_{p \in \Z}\sum_{k \in \Z} \pt{\hat{f}\,\overline{\Gapm}}\pt{\omega-k\frac{\bap}{\nu}}\PGapm(\omega) 	
\\
	&{}= \nu^{-1}\sum_{p \in \Z}\sum_{k \in \Z} \hat{f}\pt{\omega-k\frac{\bap}{\nu}}\overline{\Gapm}\pt{\omega-k\frac{\ba{p}}{\nu}}\PGapm(\omega) 	
\\
	&{}= \nu^{-1}\sum_{p \in \Z}\sum_{k \in \Z} T_{k\frac{\bap}{\nu}}\pt{\hat{f}\pt{\omega} }\overline{\Gapm}\pt{\omega-k\frac{\ba{p}}{\nu}}\PGapm(\omega), 	
\end{align*}
as desired.
\end{proof}

\begin{rem}
  We stress that the equalities \eqref{eq:direct} and \eqref{eq:Walnut} are, at this stage, formal. 
  The frame property we prove in Theorem \ref{thm:3} implies that, 
  under suitable assumption on the function $\phi$, $S_{\phi,\phi}^{\alpha;\mu, \nu}$ is a linear
  operator from $L^{2}(\R)$ to itself, therefore \eqref{eq:direct} and \eqref{eq:Walnut} are equalities in 
  $L^2(\R)$. 
\end{rem}

%%%%%%%%%%%%%%%%%%%%%%%%%%%%%%%%%%%%%%%%%%%%%%%%%%%%%%%%%%%%%%%%%%%%%%%%%%%%%%%%%%%%%%%%%%%%%%%%%%%%%%%%%%%%%%%%%%%%%%%%%%%%%
%%%%%%%%%%%%%%%%%%%%%%%%%%%%%%%%PAINLESS FRAME%%%%%%%%%%%%%%%%%%%%%%%%%%%%%%%%%%%%%%%%%%%%%%%%%%%%%%%%%%%%%%%%%%%%%%%%%%%%%%%
%%%%%%%%%%%%%%%%%%%%%%%%%%%%%%%%%%%%%%%%%%%%%%%%%%%%%%%%%%%%%%%%%%%%%%%%%%%%%%%%%%%%%%%%%%%%%%%%%%%%%%%%%%%%%%%%%%%%%%%%%%%%%
\subsection{Painless frame expansion}
In order to have a clear scheme of the frame construction we first consider the painless case.
\begin{Def}\label{d:adm}
Let $\varphi\in \S(\R)$,  such that $\mathrm{supp\,} \hat{\varphi}\subset [-L,L]$ and $0\leq \hat{\varphi}\leq 1$. Given $\mu \in \R,\:\mu >0$, we say that $(\varphi,\mu)$ is admissible if

\begin{equation}
\label{eq:G}
\Gapm(\omega) = \sum_{\eta \in \mu\Iap} \hat{\varphi}(\omega -\eta)
\end{equation}
satisfies the following properties
%%%
\normalfont{
    \begin{enumerate}
        \item[1.] $ 0\leq \Gapm(\omega)  \leq C_1$, for some $C_1>0$.
        
        \item[2.] For every $\omega \in \Rr$  $|\{p:\Gapm(\omega)>0\}| \leq C_2$.
        \item[3.] There exists $C_3 >0$ such that  for every $\omega \in \Rr^d, \; \Gapm(\omega) > C_3$, for some $p\in \Z$. 
    \end{enumerate}
}
%%%
\end{Def}

Using this definition we can immediately obtain some important properties of $\Gapm$.
\begin{lem}\label{L:3}
Let $(\varphi,\mu)$ be admissible, then 
\begin{align}
	\abs{\supp \Gapm } &\leq \pt{2L+\mu}\; \bap, \quad \forall p\in\Z,  \label{eq:support}
	\\
	 C_3^2 \leq &\sum_{p\in \Z}|\Gapm(\omega)|^2 \leq C_1^2C_2,  \quad \forall \omega \in \R, \label{eq:up_low_bund}
\end{align}
where $C_1,C_2,C_3$ are defined above (cf. Definition \ref{d:adm}).
\end{lem}
\begin{proof}
The first property comes directly from the definition of $\Gapm$ (cf. \eqref{eq:G}) and the fact that $\supp \varphi \subset \ptq{-L,L}$. 
First, we observe that  $\omega - \eta \in \supp \hat{\varphi}$ if $-L \leq \omega - \eta \leq L$; thus, writing $\eta = \mu\pt{ \iap + t}$, this translates as
%%%%
\[
\mu\iap-L \leq \omega \leq \mu\sap +L.
\]
%%%%
Hence 
%%%%
\[
\abs{\supp \Gapm } \leq \mu\pt{\sap-\iap} + 2L = \mu\bap + 2L.
\]
%%%%
Since $L$ is fixed and $\bap \geq 1$,   uniformly in $p$, we can write
%%%%
\[
\abs{\supp \Gapm }  \leq (\mu+2L)\bap,\]
%%%%
as desired.

Given $\omega \in \Rr$, we know that there are at most $C_2$ indexes for which 
$\Gapm(\omega)>0$. Moreover, $\abs{\Gapm}\leq C_1$, thus the upper bound of \eqref{eq:up_low_bund} holds true.\medskip\\
On the other hand, for the same $\omega$, there exists at least one $p$ such that $\Gapm(\omega)\geq C_3$.
\end{proof}

%%%
\begin{rem}
An explicit example of an admissible function is the compact version of a Gaussian; precisely $\hat{\varphi}(\omega) = \chi_{\epsilon}(\omega) \frac{1}{\sqrt{2}}e^{-\pi{\omega}^2}$, where
$\chi_{\epsilon}$ is a bounded smooth function such that 
\[
\chi_{\epsilon}(\omega) = \begin{cases}
	1, \qquad \omega \in \ptq{-1,1},
	\\
	0,  \qquad \omega \in \ptq{-1-\epsilon,1+\epsilon}.
\end{cases}
\]
In \cite{MR2728710}, similar functions are considered.
\end{rem}

\begin{thm}
\label{thm:3}
Consider $(\phi,\mu)$ being admissible. Then, there exist positive lattice parameters $\nu,\mu$  such that the $\alpha$-DOST system (cf. \eqref{eq:DOST_sys})
\[
D^{\alpha}\pt{\mu,\nu, \phi} = \ptg{\Bpg(x)}_{(p,k) \in \;\Z^2}
\]
is a frame for $\Ll$. Precisely, there exist $A,B>0$ such that for all $f\in \Ll$
%%%
\begin{equation}
\label{eq:Frame_P}
A\|f\|^2_2 \leq \sum_{p,k}\vert\langle f,\Bpg(x)\rangle\vert^2 \leq B\|f\|^2_2.
\end{equation}
%%%
\end{thm}
%%%
\begin{proof}
The frame property \eqref{eq:Frame_P} is  equivalent to
\[  
A\|f\|^2_2 \leq |\langle S^{\mu,\nu}_{\varphi,\varphi}f,f\rangle| \leq B\|f\|^2_2.
\]
Hence, we evaluate
\[
\Lprod{\Famnfg f} {f}=\Lprod{F(\Famnfg f)} {\hat{f}},
\]
where $F$ is the Fourier transform as above. The Walnut representation formula \eqref{eq:Walnut} yields
 \begin{align}
   \Lprod{\F\pt{\Famnfg f}} {\hat{f}}&=\frac{1}{\nu} \Lprod{\sum_{p\in\Z}\sum_{k\in \Z}T_{\frac{k}{\nu}\bap}\hat{f}\: 
   {T_{\frac{k}{\nu}\bap}\overline{\Gapm} }\Gapm}{\hat{f}}\nonumber\\
   &=\frac{1}{\nu}  \Lprod{\sum_{p\in\Z}\hat{f}\; \overline{\Gapm}\Gapm}{\hat{f}}\nonumber\\
   &\qquad+\frac{1}{\nu} 
   \Lprod{\sum_{p\in\Z}\sum_{k \in \Z\setminus \ptg{0}}T_{\frac{k}{\nu}\bap}\hat{f} \:
   {T_{\frac{k}{\nu}\bap}\overline{\Gapm}}\Gapm}{\hat{f}}.\label{eq:sepk}
    \end{align}
We notice that by \eqref{eq:up_low_bund}
%%%%
\[
\frac{1}{\nu}  \Lprod{\sum_{p\in\Z}\hat{f}\; \overline{\Gapm}\Gapm}{\hat{f}} \asymp \|f\|^2_2.
\]
%%%%
Indeed,
%%%%
\[
\frac{1}{\nu}\Lprod{ \sum_{p\in\Z} \abs{\Gapm}^2\hat{f}}{\hat{f}} \asymp \Lprod{\hat{f}}{\hat{f}} = \|f\|^2_2.
\]
%%%%
We claim that for $\nu$ small enough, 
\[
\frac{1}{\nu} 
   \Lprod{\sum_{p\in\Z}\sum_{k \in \Z\setminus \ptg{0}}T_{\frac{k}{\nu}\bap}\hat{f} \:T_{\frac{k}{\nu}\bap}\overline{\Gapm}\Gapm}{\hat{f}} = 0.
\]
Recall that $|\supp \Gapm (\omega)| \leq (2L+\mu) \bap$ (cf. \eqref{eq:support}), thus for $\nu < (2L+\mu)^{-1}$ and $k\neq 0$
\[
\supp T_{\frac{k}{\nu}\bap}\overline{\Gapm} \cap \supp \Gapm = \emptyset,
\]
as desired.

\end{proof}
%%%
\subsection{Conjugate Filter}
We can define a way to represent functions even without a canonical dual frame; namely, we construct a conjugate filter for the window function $\varphi$. This technique is a powerful tool for numerical implementations, see \cite{MR2728710}.\\
Set

\begin{equation}\label{eq:filter}
\Omega^{\mu;\nu}_{p;\alpha} (\omega) = \nu\frac{\overline{\Gapm(\omega)}}{\displaystyle\sum_{p \in \Z}
                      \abs{\Gapm(\omega)}^2 },
\end{equation}

then, the product between $\Omega^{\mu;\nu}_{p;\alpha}$ and  $\Gapm$ forms almost a partition of unity 
\begin{align}
  \label{part:un}
  \sum_{p\in \Z}\Omega^{\mu;\nu}_{p;\alpha}(\omega)\Gapm(\omega) = \nu.
\end{align}
We have the following corollary:
\begin{cor} \label{C:1}
Consider the functions $\Omega^{\alpha;\nu}_{p;\mu}$ defined in \eqref{eq:filter}, then set
\[
\widehat{\Psi}^{\alpha}_{p,k}(\omega)= \frac{1}{\sqrt{\ba{p}}}e^{-\pii\omega
                        k\frac{\nu}{\bap}
                        }
                        \overline{\Omega^{\alpha;\nu}_{p;\mu}}(\omega)
\]
Then, for any $f\in \Ll$
\begin{equation}
\label{eq:rec1}	
f(t) = \sum_{p,k} \langle f, \Psi^{\alpha}_{p,k}\rangle \Bpg(t) .
\end{equation}
\end{cor}
\begin{proof}
We notice immediately that  $\supp \Omega^{\mu;\nu}_{p;\alpha}(\omega) = \supp \Gapm(\omega)$.
We take the Fourier transform of \eqref{eq:rec1}, thus
 \begin{align}
\nonumber
		&\sum_{p,k} \langle f, \Psi^{\alpha}_{p,k}\rangle \widehat{\varphi}^{\alpha}_{p,k}(x)
\\
\nonumber
		&{}=\sum_{p\in\Z}\pt{\sum_{k\in \Z}\pt{\bap^{-\frac{1}{2}}  \Lprod{\hat{f}}{e^{-\pii\omega k
                        \frac{\nu}{\bap}
                        } \overline{\Omega^{\mu;\nu}_{p;\alpha}}(\omega)}}}
		\pt{\ba{p}^{-\tfrac{1}{2}}e^{-\pii\omega k
                        \frac{\nu}{\ba{p}}
                        }}\Gapm(\omega) \\
 \label{eq:conj} 
 &{}=\sum_{p\in\Z}\pt{ \ba{p}^{-1}\sum_{k\in \Z}  \F \pt{ \hat{f} \Omega^{\mu;\nu}_{p;\alpha} } 
                \pt{k \frac{\nu}{\bap}}
		e^{\pii\omega k
                        \frac{\nu}{\ba{p}}
                        }}
                        \Gapm(\omega).            
\end{align}
Then, by Poisson formula (cf. \eqref{eq:poiss}), \eqref{eq:conj} turns into                        
\[                                              
	\sum_{p\in\Z}\pt{ \nu^{-1}\sum_{k\in \Z}  \pt{\hat{f} \Omega^{\mu;\nu}_{p;\alpha} } 
                \pt{\omega+k\frac{\bap}{\nu}}
		}
                        \Gapm(\omega).            
\]
Repeating the same procedure as in the proof of Theorem \ref{thm:3}, i.e. excluding the terms with $k>0$,  we can conclude that, for
$\nu$ small enough,
\begin{align}
  \label{eq:quasi}
  \sum_{p,k} \langle f, \Psi^{\alpha}_{p,k}\rangle \widehat{\varphi}^{\alpha}_{p,k}(x)= 
  \frac{1}{\nu}\hat{f}(\omega)\sum_{p\in \Z} \Omega^{\mu;\nu}_{p;\alpha}(\omega)\Gapm(\omega).
\end{align}
By \eqref{part:un}, equation \eqref{eq:quasi} becomes
 \[
   \sum_{p,k} \langle f, \Psi^{\alpha}_{p,k}\rangle \widehat{\varphi}^{\alpha}_{p,k}(x)=\hat{f}(\omega).
 \]
Finally, applying the inverse Fourier transform, we obtain \eqref{eq:rec1}.

\end{proof}

\subsection{DOST frames, general construction}
In this section we prove that we can build up a $\alpha$-DOST frame with milder hypothesis on the window function compared with the ones of the previous section. \hspace{1em}\\
\begin{thm}\label{T:Frame_prop}
Consider a function $\varphi\in L^2(\Rr)\cap L^1(\Rr) $. 
As above, set $\Gapm(\omega)$ as (cf. \eqref{eq:G})
\begin{align*}
  \Gapm(\omega)= \sum_{\eta \in  \mu\Iap } \hat{\varphi}(\omega -\eta).
\end{align*}
Suppose that 
\[
a \leq \sum_{p}|\Gapm(\omega)|^2 \leq b,\quad  \forall\: \omega \in \R,
\]
for suitable constants $a,b>0$.
Moreover, assume that the window $\varphi$ satisfies
\begin{equation}
  \label{eq:decay_g}
  |\hat{\varphi}(\omega)| \leq \frac{C_N}{(1+|\omega|)^{N}}, \quad  \forall \omega \in \R
\end{equation}
for some $N>2$.

Then, there exists $A,B>0$ and $\nu_0>0$, such that for all $f \in L^2(\R)$
\[  
  A\|f\|^2_2 \leq |\langle S^{\mu,\nu}_{g,g}f,f\rangle| \leq B\|f\|^2_2,
\]
for  $0<\nu<\nu_0$.
\end{thm}
\subsection{Preparatory Lemmata}
We need some result concerning the decay of the elements $\Gapm$ outside the sets $\Iap$.
\begin{lem}\label{L:decay_G}
Let $\varphi \in \Ll \cap L^1(\Rr)$ such that:
\[
  |\hat{\varphi}(\omega)| \leq \frac{C_N}{(1+|\omega|)^{N}},
\] 
with $N> 2$. Then,
\begin{align}
  \label{eq:Dec_G}
  |\Gapm(\omega)| \leq \frac{C_{N,\mu}}{\pt{1 + \dist{\omega}{\mu \Iap}}^{N-1}}.
\end{align}
\end{lem}
\begin{proof}
From the definition of $\Gapm$ we obtain:
\[
\abs{\Gapm(\omega)} \leq \sum_{\eta \in \Iap}\abs{\hat{\varphi}(\omega-\mu \eta)} \leq \sum_{\eta \in \Iap} \frac{C_N}{\pt{1 + \abs{\omega-\mu \eta}}^N}.
\]
Hence,
\[
\abs{\Gapm(\omega)} \leq {\pt{1 + \dist{\omega}{\mu \Iap}}^{-N}}\sum_{\eta \in \Iap} \frac{C_N \pt{1 + \dist{\omega}{\mu \Iap}}^{N}}{\pt{1 + \abs{\omega-\mu \eta}}^N}.
\]
If $\omega \in \mu\Iap$, then $ \dist{\omega}{\mu \Iap}=0$. Thus
\begin{align}
	\label{eq:ineq3}
	\abs{\Gapm(\omega)} \leq \sum_{\eta\in \Iap} C_N(1 + \abs{\omega-\mu \eta})^{-N} \leq  
	C_{N} \sum_{j\in \Z} (\mu|j| +1)^{-N},
\end{align}
which is clearly convergent for $N>2$.

If $\omega \notin \mu\Iap$, then $ \dist{\omega}{\mu \Iap}=\delta$ 
and 

\begin{align*}
\abs{\Gapm(\omega)}  &{}\leq \frac{2 C_N}{(1 + \delta)^N}\sum_{|\zeta|=0}^{\bap}\frac{(1+\delta)^N}{(1 + \mu \zeta +\delta)^N}
\\
 					 &{}\leq \frac{2C_{N}}{(1 + \delta)^N}\pt{ 1+\int_0^{+\infty} \frac{(1+\delta)^N}{(1 + \mu x +\delta)^N}\ud x} 
\\ 					 
 					 &{}=  \frac{2 C_N}{(1 + \delta)^N}\pt{1+\int_0^{+\infty} \frac{1}{(1 + \frac{\mu}{1+\delta} x)^N}\ud x }
\\
 					&{}= \frac{2 C_N}{(1 + \delta)^{N}}\pt{1 + \frac{(1+\delta)}{\mu}\int_0^{+\infty} \frac{1}{(1 + x)^N}\ud x} 
\\
 					&{}= \frac{2 C_N}{(1 + \delta)^{N}}\pt{1 + \frac{(1+\delta)}{\mu}C} \leq  \frac{C_{N,\mu}}{(1 + \delta)^{N-1}}
\end{align*}
for $N> 2$, which completes the proof.
\end{proof}

\begin{lem}\label{L:unif_Sum}
Let $\omega \in \Rr$ and $N>1$. Then  
\[
\sum_{p\in \Z} \frac{1}{\pt{1 + \dist{\omega}{\mu\Iap}}^{N} }\leq C
\]
and the constant is independent on $\omega$.
\end{lem}
\begin{proof}
For any $\omega \in \Rr$, there exists only one $\overline{p}$ such that $\omega \in \mu I_{\alpha;\bar{p}}$. Then, for $j\geq 2$, 
  \begin{align*}
  	\dist{\omega}{\mu I_{\alpha;\overline{p}+ j}}&\geq \sum_{k=1}^{j-1} \beta_{\alpha}\pt{\overline{p}+ k} \geq (j-1),
  \end{align*}
  since $\bap\geq 1$. The same argument works for $j<-2$.\\
  Obviously
  \[
  \sum_{j =-1}^1 \frac{1}{\pt{1 + \dist{\omega}{\mu I_{\alpha;\overline{p}+j}}}^{N} }\leq 3.
  \]
  Hence, we can bound our sum as
  \[
    \sum_{p\in \Z} \frac{1}{\pt{1 + \dist{\omega}{\mu\Iap}}^{N} } \leq  \sum_{j\in \Z \setminus \ptg{0}} \frac{C_{N}}{j^{N} } 
  \]
and the constants do not depend on the particular choice of $\omega$, as desired.	
\end{proof}

\begin{lem}
  \label{lem:limzero}
  Let $\Gapm$ as in Theorem \ref{T:Frame_prop} above, then
  \[
    \lim_{\nu\to 0}   \sum_{p \in \Z} \sum_{k\in \Z\setminus\{0\}} \norm{\overline{\Gapm}\pt{\omega-k\frac{\ba{p}}{\nu}}\Gapm(\omega)}_{\infty}=0.
  \]
\end{lem}
\begin{proof}
  Let us define for all $p \in \Z$
\[
 _1\Gapm(\omega)= \chi_{\mu\Iap}(\omega) \Gapm(\omega)
\]
and
\[
  _2 \Gapm(\omega)= \Gapm(\omega)- \phantom{a}_1\Gapm(\omega).
\]
We can write
\begin{align}
  \nonumber
  &   \sum_{p \in \Z} \sum_{k\in \Z\setminus\{0\}}  \norm{ \overline{\Gapm}\pt{\omega-k\frac{\ba{p}}{\nu}}\Gapm(\omega)}_{\infty}
\\
  \label{eq:K}
  &\leq   \sum_{p \in \Z} \sum_{k\in \Z\setminus\{0\}}  \norm{ \overline{_1\Gapm}\pt{\omega-k\frac{\ba{p}}{\nu}} \phantom{.}_1 \Gapm(\omega)}_{\infty}
\\
  \label{eq:H}
&\quad{}+    \sum_{p \in \Z} \sum_{k\in \Z\setminus\{0\}}  \norm{\overline{_2\Gapm}\pt{\omega-k\frac{\ba{p}}{\nu}} \phantom{.}_1\Gapm\pt{\omega} }_{\infty}
\\
  \label{eq:I}&\quad{}+    \sum_{p \in \Z} \sum_{k\in \Z\setminus\{0\}}  \norm{ \overline{_1\Gapm}\pt{\omega-k\frac{\ba{p}}{\nu}}\phantom{.}_2 \Gapm(\omega)}_{\infty}
\\    \label{eq:M}
  &\quad{}+   \sum_{p \in \Z} \sum_{k\in \Z\setminus\{0\}}  \norm{\overline{_2\Gapm}\pt{\omega-k\frac{\ba{p}}{\nu}} \phantom{.}_2\Gapm(\omega)}_{\infty}.
\end{align}
For now on, let us suppose $\nu \mu<1$; then we notice that the term in \eqref{eq:K} is identically zero for each $p\in\Z$, since
the supports are disjoint. Clearly this is not restrictive since we are considering the limit for $\nu \rightarrow 0$.
We remark that:
\begin{equation}
\label{eq:distance}
	\dist{\omega -k\frac{\bap}{\nu}}{\mu \Iap} \geq \pt{\frac{|k|}{\nu}-\mu}  \bap, \qquad \forall \omega \in \mu\Iap,
\end{equation}
with
\[
  |k|-\nu\mu > 0, \quad k \in \Z\setminus\{0\}. 
\]

We then analyze the term in \eqref{eq:H}. For each $\omega$ there exists a unique $p = \overline{p}$ such that $\omega\in \mu I_{\alpha;\overline{p}}$, hence for each $\omega$
\begin{align}
 \nonumber
  & \sum_{p \in \Z} \sum_{k\in \Z\setminus\{0\}}  \abs{\overline{_2\Gapm}\pt{\omega-k\frac{\ba{p}}{\nu}} \phantom{.}_1\Gapm\pt{\omega} }
  	 = \sum_{k\in \Z\setminus\{0\}}\abs{\overline{_2 \Phi_{\overline{p};\mu}^{\alpha}}\pt{\omega-k\frac{\ba{\overline{p}}}{\nu}}  \phantom{.}_1 \Phi_{\overline{p};\mu}^{\alpha}\pt{\omega}}
 \\
 \nonumber
 & \leq C_{N,\mu}\sum_{k\in \Z\setminus\{0\}}\frac{C_{N,\mu}}{\pt{1+\dist{\omega -k\frac{\ba{\overline{p}}}{\nu}}{\mu I_{\alpha;\overline{p}}}}^{N-1}} 
 \\
 \nonumber
 & \leq C_{N,\mu}  \sum_{k\in \Z\setminus\{0\}}\frac{\tilde{C}_N}{\pt{1+ |  \pt{\frac{|k|}{\nu}-\mu}\ba{\overline{p}}|}^{N-1}}\\
 \label{eq:ultima}
 & \leq C_{N,\mu} \tilde{C}_N \sum_{k\in \Z\setminus\{0\}} \pt{\pt{\frac{|k|}{\nu}-\mu} }^{1-N}.
\end{align}
We made use of the decay proven in Lemma \ref{L:decay_G} and of \eqref{eq:distance}.
Since $N-1>1$ and the inequality \eqref{eq:ultima} does not depend on $\omega$, we obtain 
\begin{align}
\nonumber
  \sum_{p \in \Z} \sum_{k\in \Z\setminus\{0\}}  \norm{\overline{_2\Gapm}\pt{\omega-k\frac{\ba{p}}{\nu}} \phantom{.}_1\Gapm\pt{\omega} }_{\infty}&\leq C_{N,\mu} \tilde{C}_N \sum_{k\in \Z\setminus\{0\}}\pt{\pt{\frac{|k|}{\nu}-\mu}}^{1-N}\\
 \label{eq:rate}
 & = C_{N,\mu} \tilde{C}_N\nu^{N-1} \sum_{k\in \Z\setminus\{0\}}\frac{1}{\pt{|k|-\nu\mu}^{1-N}}.	
\end{align}
Since the sum is convergent uniformly with respect to $\nu$, the term in \eqref{eq:rate} goes to $0$ as $\nu \rightarrow 0$ with the rate $\nu^{N-1}$.

In order to consider the term in \eqref{eq:I}, notice that for all $\omega \in \R$, there exists
a unique $\bar{p}$ such that $\omega \in \mu I_{\alpha;\bar{p}}$. Moreover, arguing as above, 
for all $p\in \Z\setminus \ptg{\bar{p}}$ there exist a unique $k_p\in \Z\setminus\ptg{0}$ such that $\omega-
k_p\frac{\bap}{\nu} \in \Iap$. Hence, we can write \eqref{eq:I}  as
\begin{align}
  \label{eq:prim}
  \sum_{p\in \Z\setminus\ptg{\bar{p}}} \norm{ \overline{_1\Gapm}\pt{\omega-k_p\frac{\ba{p}}{\nu}}\phantom{.}_2 \Gapm(\omega)}_{\infty}.
\end{align}
By \eqref{eq:distance}
\[
  d(\omega, \mu \Iap)\geq \pt{\frac{|k_p|}{\nu}-\mu}\bap, \quad \mbox{if } \omega-k_p\frac{\bap}{\nu}\in \mu \Iap.
\]
Finally, by \eqref{eq:Dec_G}, we can write
\begin{align}
  \nonumber
  \sum_{p\in \Z\setminus\ptg{\bar{p}}} \norm{ \overline{_1\Gapm}\pt{\omega-k_p\frac{\ba{p}}{\nu}}\phantom{.}_2 \Gapm(\omega)}_{\infty}&\leq \sum_{p\in \Z\setminus\ptg{\bar{p}}} \frac{C_{N,\mu}^2}{(1+d(\omega,\mu\Iap))^{N-1}}\\
  \nonumber
  &\leq \sum_{p\in \Z\setminus\ptg{\bar{p}}} \frac{C_{N,\mu}^2}{(1+d(\omega,\mu\Iap))^{N-1-\epsilon}}\frac{1}{(1+d(\omega,\mu\Iap))^{\epsilon}} \\
  \nonumber
  &\leq \sum_{p\in \Z\setminus\ptg{\bar{p}}} \frac{C_{N,\mu}^2}{(1+d(\omega,\mu\Iap))^{N-1-\epsilon}}\frac{1}{\pt{1+\pt{\frac{|k_p|}{\nu}-\mu}\bap}^{\epsilon}}\\
  \label{eq:ultima1}
  &\leq \nu^{\epsilon} \sum_{p\in \Z\setminus\ptg{\bar{p}}}  \frac{\tilde{C}} {(1+d(\omega,\mu\Iap))^{N-1-\epsilon}}.
\end{align}
The sum in \eqref{eq:ultima1} converges uniformly with respect to $\omega$ by Lemma \ref{L:unif_Sum}, therefore the whole
sum goes to zero as $\nu$ does.

The last term to be taken into account is \eqref{eq:M}, which includes the ``tails" of our window function.\\
Fix $p\in \Z$,  by definition if either $\omega\in \mu \Iap$ or $\omega -k\frac{\bap}{\nu} \in \mu \Iap$, then 
\[
\overline{_2\Gapm}\pt{\omega-k\frac{\ba{p}}{\nu}} \phantom{.}_2\Gapm(\omega) = 0.
\]
Define the set
\[
F^{\omega}_{p} = \ptg{k : 0 < \dist{\omega -k\frac{\bap}{\nu}}{\mu\Iap} < \frac{1}{2\nu}\bap}.
\]
Notice that only a finite number of element $k$ belong to this set; precisely, when $\mu\nu<1$, then 
\[
|F^{\omega}_{p}| \leq 2, \qquad   \textrm{ uniformly in } \omega \textrm{ and } p.
\]
So, for all $\omega$, we can split the term in \eqref{eq:M} as follows:
\begin{align*}
 \sum_{p\in \Z}\sum_{k\in \Z\setminus\{0\}}\abs{\overline{_2\Gapm}\pt{\omega-k\frac{\ba{p}}{\nu}} \phantom{.}_2\Gapm(\omega)}&\leq\sum_{p\in \Z}\sum_{k\notin F^{\omega}_{p}} \abs{\overline{_2\Gapm}\pt{\omega-k\frac{\ba{p}}{\nu}} \phantom{.}_2\Gapm(\omega)}\\
&\quad + 
  \sum_{p\in \Z}\sum_{k\in F^{\omega}_{p}} \abs{\overline{_2\Gapm}\pt{\omega-k\frac{\ba{p}}{\nu}} \phantom{.}_2\Gapm(\omega)}.
\end{align*}
We show that both terms are bounded by a constant times $\nu^{\epsilon}$ with $\epsilon >$ small enough independent on $\omega$.\medskip\\
We notice that if $k\notin F^{\omega}_{p}$, then 
%%%%
\[
\dist{\omega -k\frac{\bap}{\nu}}{\mu\Iap} \geq  \frac{|k-\tilde{k}|}{\nu}\bap,
\]
%%%%
where $\tilde{k}$ is the closest index to $k$ inside $F^{\omega}_{p}$; hence, we can rearrange the (infinite) indexes $k\in \Z\setminus\{0\}$ into $j\in \Z\setminus\{0\}$ such that 
%%%%
\[
\dist{\omega -j\frac{\bap}{\nu}}{\mu\Iap} \geq  \frac{|j|}{\nu}\bap.
\] 
%%%%
From the discussion above, and the estimate \eqref{eq:Dec_G} it follows that 
%%%%
%%%%
\begin{align}
 \sum_{p\in \Z} \sum_{k\notin F^{\omega}_{p}}\abs{\overline{_2\Gapm}\pt{\omega-k\frac{\ba{p}}{\nu}}\! \phantom{.}_2\Gapm(\omega)}\nonumber &\leq  \sum_{p\in \Z}\sum_{j\in \Z\setminus\{0\}} \abs{\frac{C_{N,\mu}}{\pt{1 + \frac{|j|}{\nu}\bap}^{N-1} }\frac{C_{N,\mu}}{\pt{1 + \dist{\omega}{\mu\Iap}}^{N-1} }}\nonumber
  \\
  &{}\leq \sum_{p\in \Z} \sum_{j\in \Z\setminus\{0\}}{\frac{C_{N,\mu}}{\pt{\frac{|j|}{\nu}}^{N-1} }\frac{C_{N,\mu}}{\pt{1 + \dist{\omega}{\mu\Iap}}^{N-1} }}\nonumber
  \\
  &{}\leq \sum_{p\in \Z} \abs{ \sum_{j\in \Z\setminus\{0\}}C_{N,\mu}\pt{\frac{\nu}{j}}^{N-1}\frac{C_{N,\mu}}{\pt{1 + \dist{\omega}{\mu\Iap}}^{N-1} }} \nonumber.
  \end{align}
  The latter term is summable in $j$,   Lemma~\ref{L:unif_Sum} implies the summability in $p$ as well,
  therefore the whole term goes to zero as $\nu^{N-1}$.\medskip\\
The last part follows from the observation below:
if there exists $k \in F^{\omega}_p$, then 
%%%%
\[
0<\dist{\omega - k\frac{\bap}{\nu}}{\mu \Iap}<\frac{1}{2\nu}\bap.
\]
%%%%
Since $\mu$ is fixed and $\nu \rightarrow 0$, we can assume  $\mu\nu < 1$; thus, there exists $\delta$ such that  $\mu\nu = 1-\delta$.
We split our analysis in two cases, first we consider 
\[
  \dist{\omega - k\frac{\bap}{\nu}}{\mu \Iap} \leq \frac{\delta}{2\nu}\bap.
\]
Let $s \in \mu\Iap$ such that $\dist{\omega}{s} = \dist{\omega}{\mu \Iap}$ then by triangular inequality 
%%%%
\begin{align*}
\dist{\omega}{\mu \Iap} &{}= \dist{\omega}{s}\geq\dist{\omega}{\omega - k\frac{\bap}{\nu}} - \dist{\omega - k\frac{\bap}{\nu}}{s}\\
&{} \geq |k|\frac{\bap}{\nu} -\pt{ \frac{\delta}{2\nu} +\mu}\bap.
\end{align*}
Hence
\begin{align*}
\dist{\omega}{\mu \Iap}
&{} > |k|\frac{\bap}{\nu} - \frac{\bap}{\nu}\pt{1-\frac{\delta}{2}}= \frac{\bap}{\nu}\pt{|k| -1+ \frac{\delta}{2}}
\end{align*}
%%%%
If $|k|\geq 1$, it is clear that 
%%%%
\[
\dist{\omega}{\mu \Iap} > \frac{\delta}{2\nu}\bap
\]
%%%%
Then,
%%%% 
\[
|\phantom{.}_2\Gapm(\omega)| \leq \frac{C_{N,\mu}}{\pt{1 + \dist{\omega}{\mu\Iap}}^{N-1- \epsilon}}\frac{1}{\pt{1 + \frac{\delta}{2}
\ \frac{\bap}{\nu}}^{\epsilon}} \leq 
		\frac{C_{N,\mu} \nu^{\epsilon}}{\pt{1 + \dist{\omega}{\mu\Iap}}^{N-1- \epsilon}}.
\]
%%%%
On the other hand, when 
\[\dist{\omega - k\frac{\bap}{\nu}}{\mu \Iap} > \frac{\delta}{2\nu}\bap
\]
we have
\begin{align*}
 \abs{\overline{_2\Gapm}\pt{\omega-k\frac{\ba{p}}{\nu}}}& \leq \frac{C_{N,\mu}}{\pt{1 + \dist{\omega - k\frac{\bap}{\nu}}{\mu\Iap}}^{N-1}}\\
 & \leq \frac{C_{N,\mu}}{\pt{1+\frac{\delta}{2\nu}\beta(p)}^{N-1}}
\end{align*}
which goes to zero as $\nu^{N-1}$.
Using the inequalities, it is clear that in both cases
\begin{equation}
\label{eq:nome}
\sum_{p\in \Z} \sum_{k\in F^{\omega}_{p}}\abs{\overline{_2\Gapm}\pt{\omega-k\frac{\ba{p}}{\nu}} \phantom{.}_2\Gapm(\omega)} \leq |F^{\omega}_{p}| \sum_{p\in \Z}  \frac{\tilde{C}_N \nu^{\epsilon}}{\pt{1 + \dist{\omega}{\mu\Iap}}^{N-1- \epsilon}} 
\end{equation}
since the sum is uniformly bounded with respect to $\omega$, the above inequalities imply that 
\eqref{eq:nome} goes to zero as $\nu^{\epsilon}$. The quantity $\pt{1 + \dist{\omega}{\mu\Iap}}^{N-1- \epsilon}$ is summable if $N-1- \epsilon>1$. This is granted by the fact that $N>2$ and that we can chose $\epsilon$ small enough.

\end{proof}
\subsection{Proof of the main result}
\begin{proof}
As we did in the painless case, we can represent the frame operator as (cf. \eqref{eq:sepk})
\begin{align*}
   \Lprod{\Famnfg f} {f}&=\Lprod{\F\pt{\Famnfg f}} {\hat{f}}=\frac{1}{\nu} \Lprod{\sum_{p\in\Z}\sum_{k\in \Z}T_{k\frac{\bap}{\nu}}\pt{\hat{f}\pt{\omega} }\overline{\Gapm}\pt{\omega-k\frac{\ba{p}}{\nu}}\Gapm(\omega)}{\hat{f}\pt{\omega}}
\\
   &=\frac{1}{\nu}  \Lprod{\sum_{p\in\Z}\hat{f}\pt{\omega} \overline{\Gapm}\pt{\omega}\Gapm\pt{\omega}}{\hat{f}\pt{\omega}}
\\
   &\qquad+\frac{1}{\nu} 
   \Lprod{\sum_{p\in\Z}\sum_{k \in \Z\setminus \ptg{0}}T_{k\frac{\bap}{\nu}}\pt{\hat{f}\pt{\omega} }\overline{\Gapm}\pt{\omega-k\frac{\ba{p}}{\nu}}\Gapm(\omega)}{\hat{f}\pt{\omega}}
\\
&\leq \frac{1}{\nu}  \Lprod{\sum_{p\in\Z}\abs{\Gapm\pt{\omega}}^2\hat{f}\pt{\omega}}{\hat{f}\pt{\omega}}
\\
   &\qquad+\frac{1}{\nu} 
   \sum_{p\in\Z}\sum_{k \in \Z\setminus \ptg{0}}\norm{\overline{\Gapm}\pt{\omega-k\frac{\ba{p}}{\nu}}\Gapm(\omega)}_{\infty}\Lprod{T_{k\frac{\bap}{\nu}}\pt{\hat{f}\pt{\omega} }}{\hat{f}\pt{\omega}}
\\   
&\leq \frac{1}{\nu}  \sum_{p\in\Z}\norm{\Gapm\pt{\omega}}_{\infty}^2\|f\|_2^2\\
   &\qquad+\frac{1}{\nu} 
   \sum_{p\in\Z}\sum_{k \in \Z\setminus \ptg{0}}\norm{\overline{\Gapm}\pt{\omega-k\frac{\ba{p}}{\nu}}\Gapm(\omega)}_{\infty}\norm{T_{k\frac{\bap}{\nu}}\pt{\hat{f}\pt{\omega} }}_2 \norm{\hat{f}}_2
\\
&= \frac{1}{\nu} \ptg{ \sum_{p\in\Z}\norm{\Gapm\pt{\omega}}_{\infty}^2+ 
   \sum_{p\in\Z}\sum_{k \in \Z\setminus \ptg{0}}\norm{\overline{\Gapm}\pt{\omega-k\frac{\ba{p}}{\nu}}\Gapm(\omega)}_{\infty}} \norm{f}^2_2
\end{align*}
Similarly,
\begin{align*}
   \Lprod{\Famnfg f} {f}
&\geq \frac{1}{\nu} \ptg{ \inf_{\omega\in\R}\sum_{p\in\Z}\abs{\Gapm\pt{\omega}}^2 -
   \sum_{p\in\Z}\sum_{k \in \Z\setminus \ptg{0}}\norm{\overline{\Gapm}\pt{\omega-k\frac{\ba{p}}{\nu}}\Gapm(\omega)}_{\infty}} \norm{f}^2_2
\end{align*}
    By hypothesis, we know that
\[
    a\leq \sum_{p\in\Z}\abs{\Gapm\pt{\omega}}^2 \leq b.
\]
   Lemma \ref{lem:limzero} implies that 
\[
\lim_{\nu\rightarrow0}\sum_{p\in\Z}\sum_{k \in \Z\setminus \ptg{0}}\norm{\overline{\Gapm}\pt{\omega-k\frac{\ba{p}}{\nu}}\Gapm(\omega)}_{\infty} =0.
\]
Hence, there exists $\nu_0>0$ such that for all $0<\nu<\nu_0$
\[
         \sum_{p\in\Z}\sum_{k \in \Z\setminus \ptg{0}}\norm{\overline{\Gapm}\pt{\omega-k\frac{\ba{p}}{\nu}}\Gapm(\omega)}_{\infty}  \leq\tfrac{a}{2}
\]
Then, the action of the frame operator can be bounded as follows
\begin{align*}
	\abs{\Lprod{\Famnfg f} {f}} &\leq \tfrac{1}{\nu}\ptg{ \sum_{p\in\Z}\norm{\Gapm\pt{\omega}}_{\infty}^2+ 
   \sum_{p\in\Z}\sum_{k \in \Z\setminus \ptg{0}}\norm{\overline{\Gapm}\pt{\omega-k\frac{\ba{p}}{\nu}}\Gapm(\omega)}_{\infty}} \norm{f}^2_2\\
   &\leq \tfrac{1}{\nu}\pt{b+\tfrac{a}{2}}\norm{f}_2^2
\\
	\abs{\Lprod{\Famnfg f} {f}} &\geq \tfrac{1}{\nu}\ptg{\inf_{\omega \in \R}\sum_{p\in\Z}\abs{\Gapm\pt{\omega}}^2- 
   \sum_{p\in\Z}\sum_{k \in \Z\setminus \ptg{0}}\norm{\overline{\Gapm}\pt{\omega-k\frac{\ba{p}}{\nu}}\Gapm(\omega)}_{\infty}} \geq \tfrac{a}{2\nu}\norm{f}_2^2,
\end{align*}
as desired.
\end{proof}
\subsection{Existence of DOST frames}
We show that the Gaussian satisfies the hypotheses of the main Theorem \ref{T:Frame_prop}.
\begin{thm}\label{T:exGauss}
Consider $\varphi = \frac{1}{\sqrt{2}}e^{-\pi x^2} $ and set $\Gapm(\omega)$ as in \eqref{eq:G}.
Then, 
\begin{equation}
	\label{eq:uplow}
	a \leq \sum_{p}|\Gapm(\omega)|^2 \leq b, \forall\quad \omega \in \R.,
\end{equation}
for suitable constants $a,b>0$.
Moreover, 
\begin{equation}
\label{eq:decayrate}
|\hat{\varphi}(\omega)| \leq \frac{C_N}{(1+|\omega|)^{N}}, \quad \forall \omega \in \R,
\end{equation}
for any $N\geq 0$.	
\end{thm}
\begin{proof}
The polynomial decay claimed in~\eqref{eq:decayrate} is trivial, since $\varphi$  is a Schwartz function.
For the lower bound in \eqref{eq:uplow}, we argue as follows: for any $\omega \in \Rr$ there exists only one $\overline{p}$ such that $\omega \in \mu\Iap$. For the sake of simplicity, we assume $\overline{p}>0$, the negative case follows with the same argument.\\
Therefore,
\[
	\sum_{p}|\Gapm(\omega)|^2\geq |\Phi^{\mu}_{\overline{p};\alpha}(\omega)|^2.
\]
Since $\omega \in \mu I_{\alpha, \bar{p}}$, then $\omega = \mu(\ia{\bar{p}} + t),\: 0\leq t\leq \ba{\bar{p}}$.
Hence
\begin{align*}
|\Phi^{\mu}_{\overline{p};\alpha}(\omega)|^2 &{}= \frac{1}{2}\abs{\sum_{\eta\in \mu I_{\alpha;\overline{p}}}e^{-\pi\pt{\mu(\ia{\bar{p}} + t) -\eta}^2}}^2 
\\
&{}=\frac{1}{2} \abs{\sum_{j=0}^{\ba{\bar{p}}-1}e^{-\pi\pt{\mu(\ia{\bar{p}} + t) -\mu(\ia{\bar{p}} + j)}^2}}^2 
\\
&{}=\frac{1}{2} \abs{\sum_{j=0}^{\ba{\bar{p}}-1}e^{-\pi\mu^2(t- j)^2}}^2 
\geq\frac{1}{2} \abs{\sup_{j=0,\ldots,\ba{\bar{p}}-1}e^{-\pi\mu^2(t-j)^2}}^2,
\end{align*}
because of the positivity of the Gaussian. The maximum value is reached when  $|t-j|$ is small. 
Our construction implies that there exists $j$ such that $|t-j|\leq 1 $, thus
\begin{align*}
|\Phi^{\mu}_{\overline{p};\alpha}(\omega)|^2 &{}\geq \frac{1}{2} \abs{e^{-\pi\mu^2}}^2 = \frac{1}{2} e^{-2\pi\mu^2},
\end{align*}
which is independent on $p$ and $\omega$, as desired.\\
Due to the fact that the Gaussian is positive, it follows that also $\Gapm$ is positive as well. Hence
\[
\sum_{p}|\Gapm(\omega)|^2 \leq |\sum_{p}\Gapm(\omega)|^2.
\]
We rewrite the sum above as follows:
\[
\sum_p\Gapm(\omega) = \sum_{j}\hat{\varphi}(\omega -\mu j).
\]
Finally, since $\varphi$ belongs to the Wiener Space, we can write
\[
\sum_{p\in \Z}|\Gapm(\omega)|^2 \leq 
\pt{\textnormal{esssup} \sum_{j} \hat{\varphi}(\omega -\mu j) }^2
\leq \pt{ \pt{\frac{1}{\mu}+1} \norm{\varphi}_W}^2 < \infty,
\]
where $\norm{\cdot}_W$ is the Wiener norm. We have used a well known property of Wiener space, 
see e.g. \cite{GR01}[Lemma 6.1.2].

\end{proof}

\section{Higher Dimensions}
\label{sec:dimd}
We consider here the case $\alpha = 1$ and an arbitrary dimension. We define a (parabolic) phase space tiling and, for suitable window functions, we provide a frame of $L^2(\Rr^d)$. We follow the ideas of wave atoms proposed in \cite{MR2362408,MR2511758} and subsequently adapted to the Gaussian case by \cite{MR2728710}. For the sake of simplicity, we enlighten the notation used before by suppressing the parameter  $\alpha$.
 
As for the dimension $d=1$,  we begin with the painless case using a smooth and compactly supported window function. Moreover, we can define an explicit conjugate frame that leads to a reconstruction formula. Then we generalize the construction.
\subsection{Phase space partition}
Define the Cartesian coronae $C_p$ as follows:
%%%%
\begin{align*}
C_0 &= [-1,1]^d,	\\
C_p &= \ptg{\omega = \pt{\omega_1,\ldots,\omega_d} \in \Rr^d\, :  \max_{1\leq i\leq d} |\omega_s| \in \left[\bpp, \beta\pt{p+1}\right)},\quad p\geq 1.
\end{align*}
%%%%
Each corona is further partitioned in (open) boxes of side $\bpp=2^{p-1}$, precisely
\[
\Bapi = \prod_{s=1}^{d}\Big[\bpp\cdot \ell_s,\bpp\cdot(\ell_s + 1)\Big),
\]
where  $\ell_s=-2,-1,0,1$ and $\max_{s=1\ldots d}\pt{l_s+\frac{1}{2}}\notin C_0$, i.e. the centers are outside the inner corona. The indexes $\ell_s$ label every possible box inside the corona. It can be easily checked that the number of such boxes (or multi-indexes) is $2^{d}(2^d-1)$, for every $p\geq 1$. We also define, according to Section \ref{S:alfamod}, 
\begin{equation}
	\label{eq:Iapi}
		\Iapi = \Bapi \bigcap \Z^d,
\end{equation} 
see Figure \ref{FIG:PhSPTil}.\\
We generalize now the $1$-dimensional DOST system (cf. \eqref{eq:DOST_sys}).
\begin{Def}
Given the set $\Iapi$ (cf. \eqref{eq:Iapi}), consider
\begin{equation} \label{eq:DOST_mult}
    \Bpp(x) = \bpp^{-\tfrac{d}{2}}
                      \sum_{\eta\in\mu\Iapi}
                      e^{\pii \eta \pt{x- k\frac{\nu}{\bpp}}}
                      \varphi\pt{x- k\frac{\nu}{\bpp}}, 
\end{equation}
where $\gamma = (p,k,\ell) \in (\N,\Z^d,J)$, $J$ contains the admissible indexes $\ell$ described above. Then, we define the (multi-dimensional) DOST-system 
\[
MD\pt{\mu,\nu, \phi} = \ptg{\Bpp(x)}_{\gamma}.
\]
See Figure~\ref{F:Win2D} as an example of frame element both in time and frequency. Figure~\ref{fig:test} shows how the elements are localized in the set $\Iapi$ in frequency.
\end{Def}

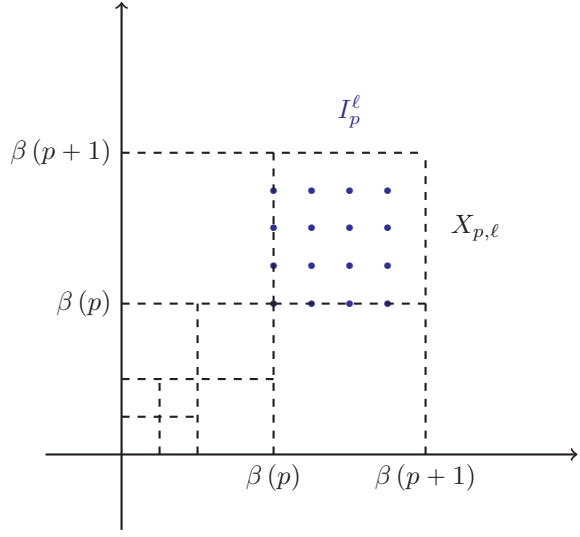
\begin{figure}
\begin{tikzpicture}[thick]
\draw[thick,->] (-1,0) -- (6,0) ;
\draw[thick,->] (0,-1) -- (0,6) ;

\draw[dashed,-] (.5,0) -- (.5,1) ;
\draw[dashed,-] (0,.5) -- (1,.5) ;

\draw[dashed,-] (1,0) -- (1,2) ;
\draw[dashed,-] (0,1) -- (2,1) ;

\node  at (2, 0) [below] {$\bpp$};
\node  at (4, 0) [below] {$\beta\pt{p+1}$};
\node  at (0, 2) [left] {$\bpp$};
\node  at (0, 4) [left] {$\beta\pt{p+1}$};

\node  at (2, 2)  {{\color{blue}\textbf{\huge.}}};
\node  at (2, 2.5)  {{\color{blue}\textbf{\huge.}}};
\node  at (2, 3)  {{\color{blue}\textbf{\huge.}}};
\node  at (2, 3.5)  {{\color{blue}\textbf{\huge.}}};
\node  at (2.5, 2)  {{\color{blue}\textbf{\huge.}}};
\node  at (2.5, 2.5)  {{\color{blue}\textbf{\huge.}}};
\node  at (2.5, 3)  {{\color{blue}\textbf{\huge.}}};
\node  at (2.5, 3.5)  {{\color{blue}\textbf{\huge.}}};
\node  at (3, 2)  {{\color{blue}\textbf{\huge.}}};
\node  at (3, 2.5)  {{\color{blue}\textbf{\huge.}}};
\node  at (3, 3)  {{\color{blue}\textbf{\huge.}}};
\node  at (3, 3.5)  {{\color{blue}\textbf{\huge.}}};
\node  at (3.5, 2)  {{\color{blue}\textbf{\huge.}}};
\node  at (3.5, 2.5)  {{\color{blue}\textbf{\huge.}}};
\node  at (3.5, 3)  {{\color{blue}\textbf{\huge.}}};
\node  at (3.5, 3.5)  {{\color{blue}\textbf{\huge.}}};

\draw[dashed,-] (2,0) -- (2,2) ;
\draw[dashed,-] (0,2) -- (2,2) ;

\draw[dashed,-] (4,0) -- (4,4) ;
\draw[dashed,-] (0,4) -- (4,4) ;

\draw[dashed,-] (2,2) -- (2,4) ;
\draw[dashed,-] (2,2) -- (4,2) ;

\node  at (3, 4.2) [above] {{\color{blue}\textbf{$\Iapi$}}};
\node  at (4.2, 3) [right] {$\Bapi$};
\end{tikzpicture}
\caption{2D phase space partitioning, and the set $\Iapi$.}
\label{FIG:PhSPTil}
\end{figure}

\subsection{Painless frame expansion}
As for the case of dimension $d=1,$ we start with compactly supported window functions. We adapt the definition of admissible window.
\begin{Def}\label{d:admd}
Let $\varphi\in \S(\R^d)$ such that $\mathrm{supp\,} \hat{\varphi}\subset [-L,L]^d$, for some $\mu>0$ and $0\leq \hat{\varphi}\leq 1$. We say that $(\varphi,\mu)$ is admissible if, given 

\begin{equation}
\label{eq:Gd}
\Gppm(\omega) = \sum_{\eta \in \mu\Iapi} \hat{\varphi}(\omega -\eta)
\end{equation}
it satisfies the following properties
%%%
\normalfont{
    \begin{enumerate}
        \item[1.] $ 0\leq \Gppm(\omega)  \leq C_1$, for some $C_1>0$.
        
        \item[2.] For every $\omega \in \Rr^d$  $|\{p,\ell:\Gppm(\omega)>0\}| \leq C_2$.
        \item[3.] There exists $C_3 >0$ such that  for every $\omega \in \Rr^d, \; \Gppm(\omega) > C_3$, for some $(p,\ell) \in \N\times J$. 
    \end{enumerate}
}
%%%
\end{Def}
Using this definition we can immediately obtain some important properties of $\Gppm$.
\begin{lem}
Let $(\varphi,\mu)$ be admissible, then 
\begin{align}
	\mathrm{Diam}\pt{\supp \Gppm }&\leq C_{L,\mu} \bpp, \qquad \forall p\in\N  \label{eq:supportd},
	\\
	 C_3^2 \leq &\sum_{p,\ell}|\Gppm(\omega)|^2 \leq C_1^2C_2, \label{eq:up_low_bundd}
\end{align}
where $C_1,C_2,C_3$ are defined above (cf. \ref{d:adm}) and $\mathrm{Diam}$ denotes the maximum distance between points of the set, i.e. the diameter.
\end{lem}
\begin{lem}
Let $\varphi, \psi$ in $\Lln$. For each $f\in \Lln$, set 
\[
\pt{S^{\nu;\mu}_{\varphi,\psi} f} (x) = \sum_{\gamma}\langle f,\Bpp \rangle \psi_\gamma(x), x\in \R^d, \gamma\in \pt{\N, \Z^d, J}.  
\]	
Then 
\begin{equation}
\label{eq:Walnutd}
\F_{x\to \omega}\pt{\PFamnfg f} (\omega)	= \frac{1}{\nu^{d}}\sum_{p,\ell}\sum_{k}T_{k\frac{\bap}{\nu}}\pt{\hat{f}\pt{\omega} }\overline{\Gppm}\pt{\omega-k\frac{\ba{p}}{\nu}}\Psi^{\mu}_{p,;\ell}(\omega), 
\end{equation}
where 
\[
\Psi^{\mu}_{p,;\ell}(\omega) = \sum_{\eta\in\mu\Iapi }
                      		\hat{\psi}(\omega- \eta).
\]
\end{lem}

The proofs of these lemmata are analogous to Lemmata~\ref{L:2}-\ref{L:3}.
%%%

\begin{thm}
Consider $(\phi,\mu)$ being admissible. Then there exists $\nu>0$ such that the DOST-system 
\[
MD\pt{\mu,\nu, \phi} 
\]
is a frame for $\Lln$. Precisely, there exist $A,B>0$ such that for all $f\in \Lln$
%%%
\begin{equation}
A\|f\|_2^2 \leq \sum_{\gamma}\vert\langle f,\Bpp(x)\rangle\vert^2 \leq B\|f\|_2^2.
\end{equation}
%%%
\end{thm}
%%%
The proof is analogous to the one made in dimension $d=1$.
%%%
\subsection{Conjugate Filter}
Set 
\begin{equation}\label{eq:filterd}
\Omega_{p;\ell}^{\mu;\nu}(\omega) = \nu^d\frac{\overline{\Gppm(\omega)}}{\displaystyle\sum_{p \in \Z}
                      \abs{\Gppm(\omega)}^2},
\end{equation}

then

\begin{equation}\label{eq:diprima}
\sum_{p,\ell}\Omega_{p;\ell}^{\mu;\nu}(\omega)\Gppm(\omega) = \nu^d.
\end{equation}

\begin{cor}
Consider the functions $\Omega_{p;\ell}^{\mu;\nu}$ defined above, set
\[
\widehat{\Psi}_{\gamma}\pt{\omega}= \bpp^{-\tfrac{d}{2}}e^{-\pii\omega
                        k\frac{\nu}{2^{p-1}}
                        }
                        \overline{\Omega_{p;\ell}^{\mu;\nu}}(\omega), \qquad \gamma = (p,k,\ell) \in (\N,\Z^d,J), \: \omega \in \R^d.
\]
Then, for any $f\in \Lln$
\begin{equation}
\label{eq:rec}	
f(x) = \sum_{\gamma} \langle f, \Psi_{\gamma}\rangle \Bpp(x) .
\end{equation}
\end{cor}
\begin{proof}
We briefly review the proof of Corollary \ref{C:1}. Nothing changes here a part from the normalization constant.
 \begin{align*}
\sum_{\gamma}& \langle \hat{f}, \widehat{\Psi}_{\gamma}\rangle \widehat{\varphi}_{\gamma}(x)
\\
		&{}=\sum_{\gamma}\pt{\bpp^{-\tfrac{d}{2}} \Lprod{\hat{f}(\omega)}{e^{-\pii\omega
                        k\frac{\nu}{\bpp}
                        }\overline{\Omega_{p;\ell}^{\mu;\nu}}(\omega)}}
			\pt{\bpp^{-\tfrac{d}{2}}e^{-\pii\omega
                        k\frac{\nu}{\bpp}
                        }\Gppm(\omega)}                      
\\
		&{}=\frac{1}{\nu^d}\sum_{p,\ell}\pt{\sum_{k\in \Z^d}\pt{\bpp^{-d}  \Lprod{\hat{f}(\omega)}{e^{-\pii\omega
                        k\frac{\nu}{\bpp}
                        }\overline{\Omega_{p;\ell}^{\mu;\nu}}(\omega)}} e^{-\pii\omega
                        k\frac{\nu}{\bpp}
                        }}\Gppm(\omega).
\end{align*}
Then, by Poisson formula (cf. \eqref{eq:poiss}) and \eqref{eq:diprima}, supposing $\nu_0$ small enough, we have for all $0<\nu<\nu_0$                        
\[                                              
	 \hat{f}(\omega)\nu^{-d}\sum_{p,\ell}\Omega_{p;\ell}^{\mu;\nu}\Gppm(\omega) = \hat{f}(\omega),
\]
as desired.
\end{proof}

\subsection{DOST frames, general construction}
We state and prove that we can build up a DOST frame with similar hypothesis to the ones given in Section \ref{sec:dim1}.
\begin{thm}\label{T:FRmult}
Consider a function $\varphi\in L^2(\Rn)\cap L^1\pt{\R^d}$. 
Let $\Gppm(\omega)$ (cf. \eqref{eq:Gd})
\begin{align*}
  \Gppm(\omega)= \sum_{\eta \in  \mu\Iapi } \hat{\varphi}(\omega -\eta).
\end{align*}
Suppose that $\mu>0$ is chosen so that 
\[
a \leq \sum_{p}|\Gppm(\omega)|^2 \leq b, \qquad\forall \omega \in \R,
\]
for suitable constants $a,b>0$.
Moreover, assume that the window $\varphi$ satisfies
\begin{equation}
\label{eq:decay_gd}
|\hat{\varphi}(\omega)| \leq \frac{C_N}{(1+|\omega|)^{N}},
\end{equation}
for some $N>2d$.

Then there exists $A,B>0$ and $\nu_0>0$, such that for all $f \in \Lln$
\[  
A\|f\|_2^2 \leq |\langle S^{\mu,\nu}_{g,g}f,f\rangle| \leq B\|f\|_2^2
\]
for some $0<\nu<\nu_0$.
\end{thm}
\subsection{Preparatory Lemmata}
We recall the same result proved in dimension $d=1$.
\begin{lem}\label{L:decay_Gd}
Let $\varphi\in \Lln\cap L^1(\Rr^d)$ such that \eqref{eq:decay_gd} holds true. Then
\begin{align}
  \label{eq:Dec_Gd}
  |\Gppm(\omega)| \leq \frac{C_N}{\pt{1 + \dist{\omega}{\mu \Iapi}}^{N-d}}.
\end{align}
\end{lem}

\begin{lem}\label{L:unif_Sum_d}
Let $\omega \in \Rr^d$ and $N>d$. Then 
\[
\sum_{p,\ell} \frac{1}{\pt{1 + \dist{\omega}{\mu\Iapi}}^{N} }\leq C
\]
and the constant is independent on $\omega$.
\end{lem}
\begin{lem}
  \label{lem:limzerod}
  Let $\Gppm$ as before then
  \[
    \lim_{\nu\to 0}  \sum_{p,\ell} \sum_{k\in \Z^d\setminus\{0\}}  \norm{\overline{\Gppm}\pt{\omega-k\frac{\bpp}{\nu}}\Gppm(\omega)}_{\infty}=0.
  \]
\end{lem}
\begin{proof}
As we did in Lemma~\ref{lem:limzero}, set
\[
 _1\Gppm(\omega)= \chi_{\mu\Iapi}(\omega) \Gppm(\omega)
\]
and
\[
  _2 \Gppm(\omega)= \Gppm(\omega)- \phantom{a}_1\Gppm(\omega).
\]

Then, we split the norm as follows:
\begin{align}
  \nonumber
  &  \sum_{p,\ell} \sum_{k\in \Z^d\setminus\{0\}}\norm{ \overline{\Gppm}\pt{\omega-k\frac{\bpp}{\nu}}\Gppm(\omega)}_{\infty}
\\
  \label{eq:Kd}
  &=  \sum_{p,\ell} \sum_{k\in \Z^d\setminus\{0\}}\norm{ \overline{_1\Gppm}\pt{\omega-k\frac{\bpp}{\nu}} \phantom{a}_1 \Gppm(\omega)}_{\infty}
\\
  \label{eq:Hd}
  &\quad{}+   \sum_{p,\ell} \sum_{k\in \Z^d\setminus\{0\}}\norm{ \overline{_2\Gppm}\pt{\omega-k\frac{\bpp}{\nu}}\phantom{a}_1 \Gppm\pt{\omega}}_{\infty}
  \\
\label{eq:Id}
&\quad{}+   \sum_{p,\ell} \sum_{k\in \Z^d\setminus\{0\}}\norm{ \overline{_1\Gppm}\pt{\omega-k\frac{\bpp}{\nu}}\phantom{a}_2 \Gppm(\omega)}_{\infty}
\\    \label{eq:Md}
  &\quad{}+  \sum_{p,\ell} \sum_{k\in \Z^d\setminus\{0\}}\norm{\overline{_2\Gppm}\pt{\omega-k\frac{\bpp}{\nu}} \phantom{a}_2\Gppm(\omega)}_{\infty}.
\end{align}
Notice that if $\nu \mu< 1$, then the term in \eqref{eq:Kd} is identically zero for each $p\in\N$, since
the supports are disjoint. 

In order to analyze the term in \eqref{eq:Hd}, notice that 
for each $\omega \in \R^d$ there exist unique $\bar{p}, \bar{l}$ such that $\omega\in \Iapibar$. Notice that we have
\begin{equation}
\label{eq:distanced}
	\dist{\omega -\frac{k\bpp}{\nu}}{\mu \Iapi} > |k|\pt{\frac{1}{\nu}-\sqrt{d}\mu}  \bpp, \qquad \forall \omega \in \Iapi.
\end{equation}
For each $\omega \in \Rr^d$, by~\eqref{eq:distanced} we have
\begin{align*}
  &\abs{ \overline{_1\Gppm}\pt{\omega}\phantom{a}_2 \Gppm\pt{\omega-k\frac{\bpp}{\nu}}} \leq \frac{C^2_N}{\pt{1+\dist{\omega}{\mu\Iapi}}^{N-d}} \pt{\pt{\frac{1}{\nu}-\sqrt{d}\mu} |k|}^{d-N}.
\end{align*}
Our hypotheses grant that $N-d>d$, then the above remarks implies that 
\begin{align*}
 \sum_{p,\ell} \sum_{k\in \Z^d\setminus\{0\}}\abs{ \overline{_1\Gppm}\pt{\omega}\phantom{a}_2 \Gppm\pt{\omega-k\frac{\bpp}{\nu}}} &\leq \abs{ \overline{_1\Gppmbar}\pt{\omega}\phantom{a}_2 \Gppmbar\pt{\omega-k\frac{\bppbar}{\nu}}}\\
 &\leq C \sum_{k\in \Z^d\setminus\{0\}}\pt{\pt{\frac{1}{\nu}-\sqrt{d}\mu}k}^{d-N}\\
 & \leq \tilde{C} \pt{\frac{1}{\nu}-\sqrt{d}\mu}^{d-N},	
\end{align*}
where the latter tends to $0$ as $\nu \rightarrow 0$ and the constants are uniformly bounded with the respect of $\omega$. Hence the term in~\eqref{eq:Hd} has a limit vanishing as $\nu$ approaches zero.

The term in \eqref{eq:Id}, goes to zero as $\nu$ goes to zero as well. 
We notice that for each $\omega$, there exist $\bar{p}, \bar{l}$ such that 
$\omega \in \Iapibar$. For the same reason for each $p,l \in \pt{\N\setminus\ptg{\bar {p}}, J\setminus \ptg{\bar{\ell}}}$, there exists
$k_{p,l}$ such that $w-k \frac{\bpp}{\nu} \in \Iapi$. Therefore, 
\begin{align}
  \nonumber
  &\sum_{p,\ell} \sum_{k\in \Z^d\setminus\{0\}}\norm{ \overline{_1\Gppm}\pt{\omega-k\frac{\bpp}{\nu}}\phantom{a}_2 \Gppm(\omega)}_{\infty}\\
  \label{eq:quasiquasi}
  &= \sum_{p\neq \bar{p},\ell\neq \bar{\ell}}\norm{ \overline{_1\Gppm}\pt{\omega-k_{p,l}\frac{\bpp}{\nu}}\phantom{a}_2 \Gppm(\omega)}_{\infty}.
\end{align}
Then, applying equation \eqref{eq:distanced}, \eqref{eq:Dec_Gd} and Lemma \ref{L:unif_Sum_d} as in \eqref{eq:ultima1} we can conclude that the term in \eqref{eq:quasiquasi} goes to zero as $\nu$ goes to zero independently on $\omega$.

We consider now the term in \eqref{eq:Md}, set
\[
  F^{\omega}_{p,\ell} = \ptg{k : 0<\dist{\omega -\frac{k\bpp}{\nu}}{\mu\Iapi} < \frac{1}{2\nu}\bpp}
\]
then
\[
  |F^{\omega}_{p,\ell}| \leq 2^d, \qquad   \textrm{ uniformly in }   \omega.
\]
One can split the term in \eqref{eq:Md} as follows:
\begin{align*}
  \sum_{k\in \Z^d\setminus\{0\}}\abs{\overline{_2\Gppm}\pt{\omega-k\frac{\bpp}{\nu}} \phantom{a}_2\Gppm(\omega)}&\leq\sum_{k\notin F^{\omega}_{p,\ell}}\abs{\overline{_2\Gppm}\pt{\omega-k\frac{\bpp}{\nu}} \phantom{a}_2\Gppm(\omega)}\\
  &\qquad + 
  \sum_{k\in F^{\omega}_{p,\ell}}\abs{\overline{_2\Gppm}\pt{\omega-k\frac{\bpp}{\nu}} \phantom{a}_2\Gppm(\omega)}.
\end{align*}
Then, for each $\omega\in \Rr^d$, 
\begin{align}
 \sum_{k\notin F^{\omega}_{p,\ell}}\abs{\overline{_2\Gppm}\pt{\omega-k\frac{\bpp}{\nu}} \phantom{a}_2\Gppm(\omega)}\nonumber \leq \frac{\nu^{d-N}C_N}{\pt{1 + \dist{\omega}{\mu\Iapi}}^{N-d}}
  \end{align}
and
\begin{equation*}
  \sum_{k\in F^{\omega}_{p,\ell}}\abs{\overline{_2\Gppm}\pt{\omega-k\frac{\bpp}{\nu}} \phantom{a}_2\Gppm(\omega)}\leq \frac{\tilde{C}_N \nu^{-\epsilon}}{\pt{1 + \dist{\omega}{\mu\Iapi}}^{N-d- \epsilon}}.
\end{equation*}
This yields
\begin{align*}
  &\sum_{p \in \Z} \sum_{k\in \Z^d\setminus\{0\}}\abs{\overline{_2\Gppm}\pt{\omega-k\frac{\bpp}{\nu}} \phantom{a}_2\Gppm(\omega)}
  \!\leq \sum_{p \in \Z} \frac{C_N \nu^{\epsilon}}{\pt{1 + \dist{\omega}{\mu\Iapi}}^{N-d- \epsilon}}
\end{align*}
which goes to zero as $\nu$ goes to zero as desired. 
We stress that $\pt{1 + \dist{\omega}{\mu\Iapi}}^{N-d- \epsilon}$ is summable if $N-d- \epsilon>d$ which  is granted by Lemma~\ref{L:unif_Sum_d}. Since the bounds are all unifrom with the respect to $\omega$, we can conclude that the terms in $\eqref{eq:Md}$ goes to zero as $\nu$ does.
\end{proof}
\subsection{Proof of the main result}
\begin{proof}
From the $d=1$ case,  we get
\begin{align*}
   \Lprod{\Famnfg f} {f}&\leq\tfrac{1}{\nu}\ptg{\sup_{\omega\in\R^d}H_0(\omega) + H^{\nu}_{k\geq1}(\omega)}\|f\|_2^2
   \\
   \Lprod{\Famnfg f} {f}&\geq\tfrac{1}{\nu}\ptg{\inf_{\omega\in\R^d}H_0(\omega) - H^{\nu}_{k\geq1}(\omega)}\|f\|_2^2
    \end{align*}
    
with 
\begin{align*}
H_0&= \sum_{p,\ell}\abs{\Gppm\pt{\omega}}^2
\\
H^{\nu}_{k\geq1}&=\frac{1}{\nu} 
   \sum_{p,\ell}\sum_{k \in \Z^d\setminus \ptg{0}}\norm{\overline{\Gppm}\pt{\omega-k\frac{\bpp}{\nu}}\Gppm(\omega)}_{\infty}.
\end{align*}
By hypothesis, we know that
    \[
    a \leq H_0(\omega) \leq b.
    \]
Lemma \ref{lem:limzerod} implies that $\lim_{\nu\rightarrow0}H^{\nu}_{k\geq1}=0$. Hence, there exists $\nu_0>0$ such that
 \[
         H^{\nu}_{k\geq1} < \tfrac{a}{2}, \quad \forall 0<\nu<\nu_0.
 \]
Then, for all $\nu\in (0, \nu_0)$, the action of the frame operator can be bounded as follows
\begin{align*}
	\abs{\Lprod{\Famnfg f} {f}} &\leq \tfrac{1}{\nu}\ptg{\norm{H_0}_{\infty} + {H^{\nu}_{k\geq1}}}\norm{f}_2^2 \leq \tfrac{1}{\nu}\pt{b+\tfrac{a}{2}}\norm{f}_2^2
\\
	\abs{\Lprod{\Famnfg f} {f}} &\geq \tfrac{1}{\nu}\ptg{\inf_{\omega\in\R^d}{H_0} - H^{\nu}_{k\geq1}}\norm{f}_2^2 \geq \tfrac{a}{2\nu}\norm{f}_2^2.
\end{align*}
\end{proof}

\begin{rem}
As we did in dimension $d=1$ (cf. Theorem~\ref{T:exGauss}), we can show that the normalized Gaussian fulfills the hypothesis of Theorem~\ref{T:FRmult}.
\end{rem}
\begin{figure}
\begin{tabular}{cc}
\subfloat[Time view, $\alpha=1,~p=3$]{\includegraphics[width=.4\linewidth]{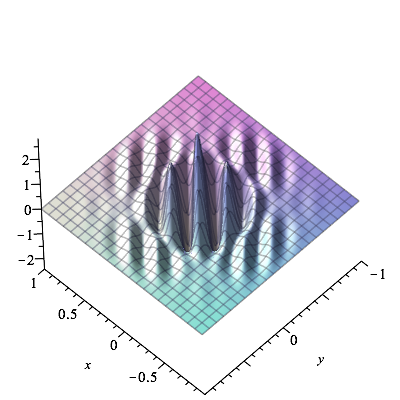}} &
\subfloat[Frequency view, $\alpha=1,~p=3$]{\includegraphics[width=.4\linewidth]{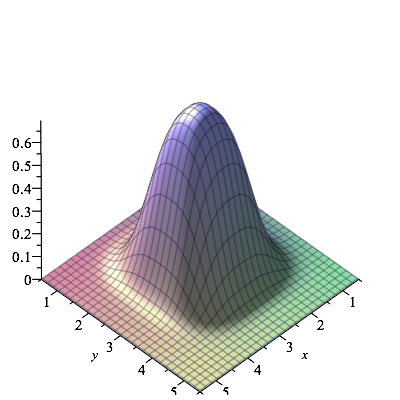}} \\
\subfloat[Time view, $\alpha=1,~p=5$]{\includegraphics[width=.4\linewidth]{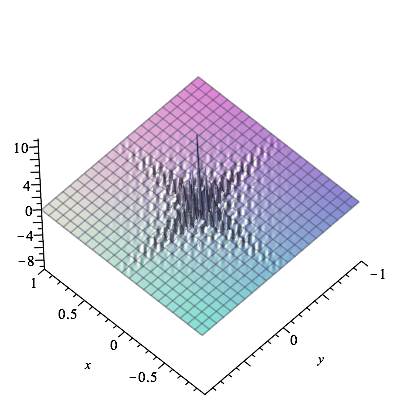}} &
\subfloat[Frequency view, $\alpha=1,~p=5$]{\includegraphics[width=.4\linewidth]{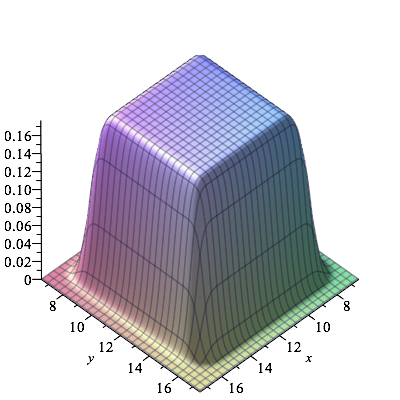}} \\
\end{tabular}
\caption{Time and frequency outlook of two window functions. We observe heavy decay in time while in frequency we localize around a certain frequency. The window $\varphi$ is in both cases a normalized Gaussian.}
\label{F:Win2D}		
\end{figure}
\begin{figure}
	\includegraphics[width=.4\linewidth]{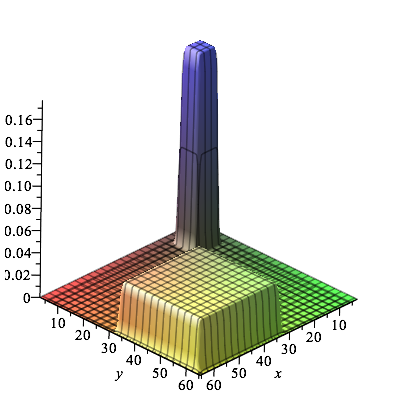}
	\caption{Two window functions for $\alpha =1$ and  $p=5,7$. We see how the normalization $\bpp^{-1}$ and the width of $\Iapi$ affects the shape of the windows. The window $\varphi$ is, in both cases, a normalized Gaussian.}
\label{fig:test}		
\end{figure}

\section{Conclusions}
We constructed a frequency-adapted frame which covers  Gabor and Stockwell-related frames. Our setting includes also general $\alpha$-phase-space partitioning. This approach appears natural to describe $\alpha$-Modulation spaces and this will be subject of a future work.\\
\indent In \cite{Battisti2015}, the author prove that the DOST basis is able to diagonalize the $S$-transform
with a suitable window function which is essentially a boxcar window in the frequency domain and that the evaluation
of the DOST-coefficients turns out to be the evaluation of the $S$-transform with this particular window in a suitable lattice.
The natural question is if the $\alpha$-DOST bases introduced in this paper have the same property, clearly not with respect to
the $S$-transform, but in relation to another transform, similar to the \emph{flexible Gabor-wavelet transform} (or
\emph{$\alpha$-transform}), see e.g. \cite{Fornasier2007157} for the the definition.\\
\indent The $n$-dimensional case considered in Section \ref{sec:dimd} is restricted to the case $\alpha=1$, hence a suitable phase-space partitioning is yet to be defined for $\alpha \in [0,1)$ and will be part of our future studies. This issue has been
already analyzed in the two dimensional case by N. Morten in \cite{Mo10} using the theory of Brushlets. \\
\indent From a computational stand point, we aim to implement and compare our results with existent methods. 
We are interested in testing in various applications such as medical and seismic imaging and also general image processing. 
As pointed out in the introduction, we remark that our approach consider the $n$-dimensional case in a peculiar way: instead of applying the one dimensional DOST in each direction sequentially, we provide a native $n$-dimensional setting. This approach is similar to the Wavepackets and Curvelets one, see \cite{MR2362408,MR2728710}.\\
\indent A natural question arises: is the density of our frames comparable with the Gabor case? For instance, is it true that if the volume of the lattice is strictly lower than $1$, the Gaussian leads to a frame? And is this condition independent on $\alpha?$\hspace{1em}\\
\paragraph{\bf{Acknowledgments}}
We thank Elena Cordero, Fabio Nicola, Luigi Riba and Anita Tabacco  for the useful discussions on the subject. We are also graceful to Maarten V. de Hoop and Man Wah Wong for the opportunity to talk about our projects in international meetings.\\
The authors were partially supported by  the Gruppo
Nazionale per l'Analisi Matematica, la Probabilit\`a e le loro
Applicazioni (GNAMPA) of the Istituto Nazionale di Alta Matematica
(INdAM). The first author is partially supported by the Research Project FIR (Futuro in Ricerca) 2013 \emph{Geometrical and qualitative aspects of PDE's}.

\bibliography{Bib_Stock}
%\nocite{*}
\bibliographystyle{siam} 

\end{document}